\documentclass[10pt]{amsart}
\usepackage{eurosym}
\usepackage{hyperref}
\usepackage[bottom]{footmisc}
\usepackage{geometry}
\usepackage{xcolor}
\usepackage{graphicx}
\usepackage{enumitem}
\usepackage{amssymb}
\usepackage{mathrsfs}
\usepackage{amsthm}

\newtheorem{theoremn}{Theorem}[section]
\newtheorem{lemman}[theoremn]{Lemma}
\newtheorem{propositionn}[theoremn]{Proposition}

\newtheorem{corollaryn}[theoremn]{Corollary}

\newtheorem*{theorem*}{Theorem}
\newtheorem*{lemma*}{Lemma}
\theoremstyle{definition}
\newtheorem{definition}[theoremn]{Definition}

\newtheorem{remark}[theoremn]{Remark}
\newtheorem{example}[theoremn]{Example}
\newtheoremstyle{definition*}
{\topsep}
{\topsep}
{}
{0pt}
{\bfseries}
{.}
{ }
{\thmname{#1}\thmnumber{ #2}\thmnote{ (#3)}}
\theoremstyle{definition*}
\newtheorem*{definition*}{Definition}
\newtheorem*{remark*}{Remark}
\newtheorem*{claim*}{Claim}
\newtheorem*{corollary*}{Corollary}

\begin{document}
\title[Commuting Graph of a group action with few edges]{Commuting Graph of
a Group Action\\
with few edges}
\author{\.{I}sma\.{I}l \c{S}. G\"{u}lo\u{g}lu}
\address{\.{I}sma\.{I}l \c{S}. G\"{u}lo\u{g}lu, Department of Mathematics, Do%
	\u{g}u\c{s} University, Istanbul, Turkey}
\email{iguloglu@dogus.edu.tr}
\author{G\"{u}l\.{I}n Ercan$^{*}$}
\address{G\"{u}l\.{I}n Ercan, Department of Mathematics, Middle East
	Technical University, Ankara, Turkey}
\email{ercan@metu.edu.tr}
\thanks{$^{*}$Corresponding author}
\maketitle

\begin{abstract}
Let $A$ be a group acting by automorphisms on the group $G.$ \textit{The
commuting graph $\Gamma(G,A)$ of $A$-orbits} of this action is the simple
graph with vertex set $\{x^{A} : 1\ne x \in G \}$, the set of all $A$-orbits
on $G\setminus \{1\}$, where two distinct vertices $x^{A}$ and $y^{A}$ are
joined by an edge if and only if there exist $x_{1}\in x^{A}$ and $y_{1}\in
y^{A}$ such that $[x_{1},y_{1}]=1$. The present paper characterizes the
groups $G$ for which $\Gamma(G,A)$ is an $\mathcal{F}$-graph, that is, a
connected graph which contains at most one vertex whose degree is not less
than three.
\end{abstract}

\baselineskip12pt


\section{introduction}

Throughout this paper group means a finite group. A great deal is known on
deriving information about the structure of a group $G$ from some certain
properties of an associated graph. The commuting graph of a group has been
one of the most popular amongst such graphs; and several results indicating
the influence of the commutativity relation on the structure of a group have
been obtained. As a generalization, the commuting graph of conjugacy classes
has been introduced and analyzed in \cite{Lo}. Recently, focusing attention
on a further generalization, the concept of the commuting graph of a group
action has been introduced in \cite{IG} as follows.

\begin{definition}
Let $A$ be a group acting by automorphisms on the group $G.$ \textit{The
commuting graph $\Gamma(G,A)$ of $A$-orbits} is the graph with vertex set $%
\{x^{A} : 1\ne x \in G \}$, the set of all $A$-orbits on $G\setminus \{1\}$,
where two distinct vertices $x^{A}$ and $y^{A}$ are joined by an edge if and
only if there exist $x_{1}\in x^{A}$ and $y_{1}\in y^{A}$ such that $x_{1}$
and $y_{1}$ commute.
\end{definition}

In \cite{IG} the connectedness of $\Gamma(G,A)$ has been studied and the
structure of the group $G$ has been investigated in cases where $\Gamma(G,A)$
is complete or triangle free or contains a complete vertex or contains an
isolated vertex. We would like to mention here only the fact that $G$ is nilpotent if 
$\Gamma(G,A)$ is complete  (see Theorem 3.1 in 
\cite{IG}) due to its appearance in the next sections.

Of course, for further study on $\Gamma(G,A)$ several problems can be
suggested. In the present article, as a matter of taste, we want to
characterize all groups $G$ admitting a group $A$ of automorphisms so that $%
\Gamma(G,A)$ has few edges. More precisely we fix the notion of being a
graph ``having few edges" as being an $\mathcal{F}$-graph in the following
sense.

\begin{definition}
A finite simple graph $\Gamma$ is an \textbf{$\mathcal{F}$-graph} if it is
connected and contains at most one vertex whose degree is not less than
three. If an $\mathcal{F}$-graph $\Gamma$ contains a vertex $v$ with more
than two neighbours then $v$ is uniquely determined by this property and is
called \textit{\ \textbf{the singular vertex}} of $\Gamma.$
\end{definition}

All path-graphs $P_{n}$ and cycles $C_{n}$ with $n$ vertices are $\mathcal{F}
$-graphs. Friendship graphs and starlike trees are
special $\mathcal{F}$-graphs.

The purpose of the present paper is to describe the groups $G$ which has a
group $A$ of automorphisms such that $\Gamma(G,A)$ is an $\mathcal{F}$%
-graph. The main result we obtained is as follows.\newline

\begin{theorem*}\label{Theorem}
Let $A$ be a group acting by automorphisms on the group $G$ with $\left\vert
\pi(G)\right\vert \geq2$ so that the commuting graph $\Gamma(G,A)$ of $A$%
-orbits is an $\mathcal{F}$-graph. Then either $\Gamma(G,A)$ does not
possess a singular vertex and $G=P\times Q$ where $P$ and $Q$ are elementary
abelian $p$- and $q$-groups for some distinct primes $p $ and $q$,
respectively, and $P\setminus \{1\},Q\setminus \{1\}$ and $G\setminus (P\cup
Q)$ are all the distinct $A$-orbits; or $\Gamma(G,A)$ does have a uniquely
determined singular vertex $z^A$. Then $|z|$ is a prime $p$ and one of the
following holds:\newline

$(1)$ $O_{p}(G)=1$ and $p=2$ and $G=F(G)S$ where $S\in Syl_2(G)$ and $F(G)$
is elementary abelian. Moreover one of the following
holds where $H=GA$ and $\overline{H}=H/C_{H}(F(G))$.

\begin{enumerate}
\item[(a)]  $S\cong D_{8}$, $\left\vert \overline{H}/\overline{G}\right\vert
=q-1 $ and $\left\vert F(G)\right\vert =q^{2}$ where $q\in \{3,7\}$.
\item[(b)] $S\cong Q_{8}\ast D_{8}$, the extraspecial group of order $32$
with $20$ elements of order $4$, $\left\vert F(G)\right\vert =3^{4}$, and $%
\overline{H}/\overline{G}$ is a Frobenius group of order $10$ or $20.$
\end{enumerate}

$(2)$ $1\neq O_{p}(G)=F(G)$ and $C_{G}(F(G))\leq F(G)$ and for $\overline{G}%
=G/F(G)$ one of the following holds:

\begin{enumerate}
\item[(a)] $\overline{G}$ has a normal subgroup of prime order which is
complemented by a (possibly trivial) cyclic $p$-group of order dividing $%
p^{2}$ and acting Frobeniusly on it.

\item[(b)] $\overline{G}\cong SL(2,4)$,$\ Z(G)=z^{A}\cup\{1\},\; 1\neq
F(G)/Z(G)$ is a direct sum of natural irreducible modules of $\overline{G}.$
Furthermore $G$ does not split over $Z(G)$.

\item[(c)] $\overline{G}\cong SL(2,4),\; p=5,\; F(G)$ is elementary abelian, 
$F(G)\setminus \{1\}$ is the union of two $A$-orbits, all $p$-elements in $%
G\setminus F(G)$ lie in the same $A$-orbit, $z^A\cap Z(S)=\phi $ for any $%
S\in Syl_p(G) $, and $A/C_{A}(\overline{G})$ is isomorphic to $S_{4}$ or $%
S_{5}$.
\end{enumerate}

$(3)$ $1\neq O_{p}(G)=F(G)$ and $C_{G}(F(G))\not \leq F(G)$ where either $G$
is a quasisimple group which is isomorphic to one of $SL(2,5),\; 2PSL(3,4),\;
2^{2}PSL(3,4)$, or $G=O_{p}(G)\times E(G)$ with $O_{p}(G)=\langle
z^{A}\rangle $ and $E(G)\cong PSL(2,5).$ \

$(4)$ $1\neq O_{p}(G)\neq F(G)$ such that $F(G)=P\times Q$ where $P=O_{p}(G)$
is an elementary abelian $p$-group, $Q=O_{p^{\prime}}(F(G))$ is a Sylow $q$%
-subgroup of $G$ for another prime $q$ which is elementary abelian, $G/P$ is
a Frobenius group with kernel $F(G)/P$ and complement either of prime order
or a $p$-group which has a unique subgroup of order $p $. Furthermore both $%
P\setminus \{1\}$ and $Q\setminus \{1\}$ are $A$-orbits.
\end{theorem*}

This paper is divided into five sections. Section 2 investigates the
structure of the group $G$ when $\Gamma (G,A)$ contains no singular vertex.
Section 3 is devoted to the completion of the proof of the theorem above in
case $\Gamma (G,A)$ contains a singular vertex. This section is divided into
four subsections the first of which includes several basic observations
arising from the existence of a singular vertex to most of which we appeal
frequently throughout the rest of the paper, while the second presents some
critical groups $G$ for which $\Gamma (G,A)$ is not an $\mathcal{F}$-graph
for any $A\leq Aut(G)$. The other subsections study the cases $O_{p}(G)=1$
and $O_{p}(G)\neq 1$ separately and contain several examples. The paper
ends with Section 4 including some final remarks on some immediate
consequences of the Theorem\ref{Theorem}.

Before closing this introduction we want to remind the reader of two basic well-known facts which we shall use repeatedly without explicit reference:

	\begin{enumerate}
		\item If $A$ is a noncyclic, elementary abelian group acting coprimely by automorphisms
		 on the group $G$ then $G=\left\langle C_{G}(a):1\neq a\in
		A\right\rangle .$
		
		\item If $A=FH$ is a Frobenius group with kernel $F$ and complement $H,$ acting
		on the group $G$ by automorphisms so that $C_G(F)=1,$ then
		$C_{G}(H)\neq1.$
	\end{enumerate}

For all the properties of the simple groups appearing in this paper we refer
to Atlas \cite{Atlas} without mentioning it explicitly.

Throughout, for any graph $\Lambda$ with vertex set $V,$ the induced graph
on $W\subset V$ will be denoted by $\Lambda[W]$. If $W=\{v_{1},\ldots,v_{n}%
\} $ we simply write $\Lambda[v_{1},\ldots,v_{n}]$ rather than $\Lambda[%
\{v_{1},\ldots,v_{n}\}]$ for $\Lambda[W]$.

To simplify the notation, throughout $\Gamma $ will denote the commuting
graph $\Gamma (G,A)$ of $A$-orbits of the finite group $G.$

\section{\textbf{When $\Gamma$ is an $\mathcal{F}$-graph without singular
vertex}}

We begin by examining the structure of $G$ in case where $\Gamma$ is an $%
\mathcal{F}$-graph with no singular vertex.

\begin{propositionn}
\label{Propo 1} Suppose that $\Gamma$ has no singular vertex. Then

\begin{enumerate}
\item[(i)] $\Gamma$ is either $P_{n}$ with $n\leq 3$ or $C_{3} $;

\item[(ii)] Either $G$ is a $p$-group for some prime $p$, or $\Gamma=C_{3}$
and $G=P\times Q$ where $P$ and $Q$ are elementary abelian $p$- and $q$%
-groups for some distinct primes $p $ and $q$, respectively. Furthermore, $%
P\setminus \{1\},Q\setminus \{1\}$ and $G\setminus (P\cup Q)$ are all the
distinct $A$-orbits.
\end{enumerate}
\end{propositionn}

\begin{proof}
$(i)$ An $\mathcal{F}$-graph with no singular vertex is a connected graph in
which every vertex is of degree at most $2$ and hence it is either $P_{n}$
for some positive integer $n$ or $C_{m}$ with $m\geq3$. If $G$ is a $p$%
-group for a prime $p$ then any vertex $x^{A}$ for some $x\in Z(G)$ is a
complete vertex with $\deg(x^{A})\leq2$ which yields that $\left\vert
V(\Gamma)\right\vert \leq3$ proving the claim. Assume next that $\left\vert
\pi(G)\right\vert \geq2$ and every element is of prime power order. As the
graph is connected there must exist two adjacent vertices $x^{A}$ and $y^{A}$
represented by elements of coprime orders which is impossible. Finally
assume that there exists $x\in G$ the order of which is divisible by two
distinct primes. Then $x^{A}$ is a vertex of a subgraph isomorphic to $C_{3}$
whence $\Gamma=C_{3}$ which completes the proof of $(i)$. \medskip

$(ii)$ Assume that $\left\vert \pi(G)\right\vert \geq2$. An argument in the
proof of $(i)$ shows that $\Gamma= C_{3}$ is a complete graph. It follows by
Theorem 3.1 of \cite{IG} that the group $G$ is nilpotent. Let $p$ and $q$ be
two distinct elements of $\pi(G)$; and $x$ and $y$ be elements of $G$ of
orders $p$ and $q$, respectively. Since $x$ and $y$ commute it holds that $%
\Gamma[x^{A} ,y^{A},(xy)^{A}]=C_{3}=\Gamma$, that is, $V(\Gamma)=\{x^{A}
,y^{A},(xy)^{A}\}$. In particular, $\pi(G)=\{p,q\}$, $x^{A}$ is the set of
elements of $G$ of order $p$, $y^{A}$ is the set of elements of order $q$
and all the other nonidentity elements of $G$ are of order $pq$ forming the
orbit $(xy)^{A}$.

Let $P$ be the Sylow $p$-subgroup of $G.$ Clearly $P\setminus \{1\}=x^{A}$.
We may assume that $x\in \Omega _{1}(Z(P))$. It follows that $P$ is
elementary abelian. Similarly the Sylow $q$-subgroup $Q$ of $G$ is
elementary abelian and $y^{A}=Q\setminus \{1\}$. Then the set $\{uv:u\in
P\setminus \{1\},v\in Q\setminus \{1\}\}=(xy)^{A}$ consists of elements of $%
G $ of order $pq$. This establishes $(ii)$.
\end{proof}

\begin{remark}
Observe that for any two distinct prime numbers $p$ and $q$ and any
elementary abelian $p$-group $P$ and any elementary abelian $q$-group $Q$
there exists a subgroup $A$ of $Aut(P\times Q)$ such that $\Gamma (P\times
Q,A)=C_{3}.$
\end{remark}

\section{\textbf{When $\Gamma $ is an $\mathcal{F}$-graph with singular
vertex}}

The following precise description of Eppo-groups (or $CP$-groups by some
authors), namely, groups in which every element is of prime power order will
be frequently used in the rest of the paper.

\begin{theoremn}[Main Theorem of \protect\cite{Del}]
\label{Theo 1} One of the following holds for any Eppo-group $E$ with $%
\left\vert \pi(E)\right\vert \geq2$:

\begin{enumerate}
\item[(a)] $E$ is a Frobenius group with $\left\vert \pi(E)\right\vert =2;$

\item[(b)] $E$ is a $2$-Frobenius group with $\left\vert \pi(E)\right\vert
=2;$

\item[(c)] $E$ is isomorphic to one of the following groups:\newline
$PSL(2,q)$ for $q\in\{5,7,8,9,17\}$, $PSL(3,4)$, $Sz(8)$, $Sz(32)$, $%
M_{10}$;

\item[(d)] $O_{2}(E)\neq1$ and $E/O_{2}(E)$ is isomorphic to one of the
following groups: \newline
$PSL(2,q)$ for $q\in\{5,8\}$,  $Sz(8)$, $Sz(32).$\newline
Furthermore $O_{2}(E)$ is isomorphic to a direct sum of natural irreducible
modules for $E/O_{2}(E).$
\end{enumerate}
\end{theoremn}

Henceforth we shall concentrate on the question of what information can be
deduced about the group $G$ when $\Gamma$ is an $\mathcal{F}$-graph with
singular vertex. All results of this section are obtained under the
hypothesis below without explicitly mentioning it in each case. \newline

\textbf{\textit{Hypothesis.}} \textit{\, Suppose that $\left\vert
\pi(G)\right\vert \geq2$ and that $\Gamma=\Gamma(G,A)$ is an $\mathcal{F}$%
-graph with the singular vertex $z^{A}$. Fix $p$ as one of the prime
divisors of $|z|$.}\newline

\subsection{\textbf{KEY OBSERVATIONS}}

\begin{propositionn}
\label{Propo 2} The following hold for $\Gamma $.

\begin{itemize}
\item [(a)] We have $|z|=p$.\medskip

\item[(b)] Every connected component $\Delta $ of $\Gamma \lbrack V\setminus
\{z^{A}\}]$ is isomorphic to $P_{n}$ for some positive integer $n.$
One of the pendant vertices of $\Delta $ is connected to $z^{A}$ in $\Gamma$ and the only vertices of $\Delta $ which are connected to $z^{A}$ in $\Gamma $ are contained in the set of pendant vertices of $\Delta .$\medskip

\item[(c)] If $x^{A}\sim y^{A}$ with $(|x|,|y|)=1$ then $x$ and $y$ are of
prime orders and $z^{A}\in \{x^{A},y^{A}\}.$ In particular, for any $q$%
-element $u$ where $q\neq p$, the $C_{G}(u)$ is a $\{p,q\}$-group and for
any $p$-element $v$ such that $v^{A}\neq z^{A}$, $C_{G}(v)$ is a $p$-group.\medskip

\item[(d)] Let $M$ and $N$ be $A$-invariant subgroups of $G$ such that $%
N\vartriangleleft M$. If $z^{A}\cap M\subset N,$ then $M/N$ is an
Eppo-group. In particular, $G/N$ is an Eppo-group if $z^{A}\cap N\neq
\emptyset $ and $N$ is an Eppo-group if $z^{A}\cap N=\emptyset $.\medskip

\item[(e)] If $z^{A}$ is a complete vertex of $\Gamma $ then either $A$ is
not contained in $Inn(G)$ or $z\in Z(G)$ and $G/Z(G)$ is an Eppo-group.
Furthermore $\Gamma $ consists of a certain number of $C_{3}$'s and a
certain number of $P_{2}$'s joined at $z^{A}$.\medskip

\item[(f)] If $\Delta $ is a connected component of $\Gamma \lbrack
V\setminus \{z^{A}\}]$ which is not a subgraph of a triangle in $\Gamma $,
then $|x|$ divides $p^{2}$ for each $x^{A}\in V(\Delta )$. In particular,
the exponent of a Sylow $p$-subgroup of $G$ divides $p^{3}$.\medskip

\item[(g)] Let $Q$ be a Sylow $q$-subgroup of $G$ for a prime $q\neq p$ and
let $1\neq x\in \Omega _{1}(Z(Q))$. Then $x^{A}\sim z^{A}$, $Q\setminus
\{1\}\subseteq x^{A}$ and hence $exp(Q)=q.$ In particular, any two commuting
nontrivial $q$-elements lie in the same $A$-orbit.\medskip

\item[(h)] Let $H$ be an $A$-invariant subgroup of $G$ and $q$ be a prime
different from $p$. Then for any $Q\in Syl_{q}(G)$, either $Q\leq H$ or $%
Q\cap H=1.$ In particular, $(\left\vert H\right\vert ,[G:H])$ is a power of $%
p.$\medskip

\item[(i)] The Gr\"{u}nberg-Kegel graph of $G$ is the complete binary graph $%
K_{1,n}$ where $n+1=\left\vert \pi (G)\right\vert .$\medskip

\item[(j)] $(z^{x})^{A}=z^{A}$ for any $x\in GA$. In particular $z^{A}$ is
invariant under conjugation by elements of $G$, that is, $z^{A}$ is a normal
subset of $G$. \medskip

\item[(k)] If there exists a Sylow $p$-subgroup $P$ such that $z\in Z(P)$
then $z^{A}$ is a complete vertex of $\Gamma $ and $z^{A}\subset O_{p}(G)$%
.\medskip

\item[(l)] The distance of a pendant vertex from $z^{A}$ is at most
two.\medskip

\item[(m)] If $C_{n}$ appears as the subgraph of $\Gamma $ then it contains $%
z^{A}$ as a vertex and $n\leq 4.$\medskip
\end{itemize}
\end{propositionn}

\begin{proof}
$(a)$ Suppose that $z^{p}\ne 1$. Then $(z^{p})^{A}$ is a vertex different
from $z^{A}$ which is adjacent to each vertex adjacent to $z^{A}.$ This
leads to the contradiction $\deg((z^{p})^{A})\geqq\deg(z^{A})>2$. Therefore $%
|z|=p$ as claimed.\medskip

$(b)$ Let $\Delta $ be a connected component of $\Gamma \lbrack V\setminus
\{z^{A}\}]$. It follows by a similar argument as in the proof of Proposition %
\ref{Propo 1} $(i)$ that $\Delta =P_{n}$ for some positive integer $n$ or $%
C_{m}$ with $m\geq 3.$ The latter is impossible as $\Gamma $ is connected
and $z^{A}$ is the only vertex of $\Gamma $ of degree greater than two.
Because of the same reason the only vertices of $\Delta $ which are connected to $z^{A}$ in $\Gamma $ are contained in the set of pendant
vertices of $\Delta .$ 
Clearly at least one of the pendant vertices of $%
\Delta $ is connected in $\Gamma $ to $z^{A}.$\medskip

$(c)$ Assume that $x^{A}\sim y^{A}$ and that $(|x|,|y|)=1$. Without loss of
generality we may assume that $x$ and $y$ commute. Then $\Gamma \lbrack
x^{A},y^{A},(xy)^{A}]=C_{3}$ and hence is not a subgraph of $\Gamma \lbrack
V\setminus \{z^{A}\}]$ by $(b)$. It follows that $z^{A}\in $ $%
\{x^{A},y^{A},(xy)^{A}\}.$ Without loss of generality we may assume that $%
x^{A}=z^{A}$. Let $q$ be a prime divisor of $|y|.$ If $|y|\neq q$ then we
get $\Gamma \lbrack x^{A},y^{A},(xy)^{A},(y^{q})^{A}]=K_{4}$ which is also
impossible by paragraph $(b)$. Let now $u$ be a $q$-element where $q\neq p$
and let $w\in C_{G}(u)$ be an $r$-element for some $r\neq q$. Then $%
u^{A}\sim w^{A}$, and so $r=p$ by the above argument. Next let $v$ be a $p$%
-element such that $v^{A}\neq z^{A}.$ Clearly $v$ cannot be centralized by
an element the order of which is divisible by a prime different from $p$.\medskip

$(d)$ If $z^{A}\cap M\subset N,$ and there exists an element $xN$ of order $%
rs$ for two distinct primes $r$ and $s$ in $M/N$. Since $\left\langle
x\right\rangle $ contains an element of order $rs$ we observe by $(c)$ that $%
p\in \{r,s\}$ and the $p$-part $u$ of $x$ belongs to $z^{A}$. This leads to
a contradiction as $z^{A}\subset N$ but $u\notin N$. Therefore every
nontrivial element of $M/N$ has prime power order, that is $M/N$ is an
Eppo-group. \newline

$(e)$ That $z^{A}$ is a complete vertex means that for any nonidentity $x\in
G$ there exists some $a\in A$ such that $[z^{a},x]=1,$ that is, $\bigcup
_{a\in A} C_{G}(z)^{a}=G.$ If $A\leq Inn(G)$ then $\bigcup _{a\in A}
C_{G}(z)^{a}\subseteq \bigcup_{g\in G} C_{G}(z)^{g}$ which yields that $%
C_{G}(z)=G.$ Thus one can immediately conclude by $(d)$ that $G/Z(G)$ is an
Eppo-group.\newline

$(f)$ Let $\Delta $ be a connected component of $\Gamma \lbrack V\setminus
\{z^{A}\}]$ which is not a subgraph of a $C_{3}$ in $\Gamma $ and pick $%
x^{A} $ from $V(\Delta )$. Suppose that $|x|$ is divisible by two distinct
primes $r$ and $s$. Then there are $x_{1}$ and $x_{2}$ in $\left\langle
x\right\rangle $ of orders $r$ and $s$, respectively. Then $\Gamma \lbrack
x_{1}^{A},x_{2}^{A},x^{A}]=C_{3}$ is a subgraph of $\Gamma $ which is not
possible$.$ This shows that $|x|$ is a power of a prime $q$. In fact every
vertex in $\Delta $ must have a representative having order which is a power
of the same prime $q$ because otherwise $z^{A}$ appears as a vertex in $%
\Delta $ by $(c)$. Let $u^{A}$ be a vertex of $\Delta $ which is adjacent to 
$z^{A}.$ Without loss of generality we may assume that $u$ and $z$ commute.
If $q\neq p$ then $(uz)^{A}\in V(\Delta )$ which is impossible. Thus we have 
$q=p.$ Furthermore, if $p^{3}$ divides $|u|,$ we see that $\Gamma \lbrack
u^{A},(u^{p})^{A},(u^{p^{2}})^{A}]=C_{3}$ and hence $(u^{p^{2}})^{A}=z^{A}$
by $(b)$ and $(a)$. In particular $|u|=p^{3}$ completing the proof.\medskip

$(g)$ Let $Q$ be a Sylow $q$-subgroup of $G$ for some $q\ne p$ and let $1\ne
x\in\Omega_{1}(Z(Q)).$ Suppose that $x^{A}\ne y^{A}$ for some $y\in
Q\setminus \{1\}$. It holds that $x^{A}\sim y^{A}$. Let $\Delta$ denote the
connected component of $\Gamma[V\setminus \{z^{A}\}]$ containing $x^A$ and $%
y^A$. As $p\ne q$, we have $\left\vert \Delta\right\vert =2$ by $(f)$. We
may assume that $x^{A}\sim z^{A}$ and that $x$ commutes with $z.$ It follows
now that $(xz)^A\in V(\Delta)$ which is a contradiction. Therefore $%
Q\setminus \{1\}\subseteq x^A$ and hence $\exp(Q)=q$.\medskip

$(h)$ Let $H$ be an $A$-invariant subgroup of $G,$ and let $Q$ be a Sylow $q$%
-subgroup of $G$ for $q\in\pi(G)\setminus \{p\}$ such that $Q\cap H\neq1.$
Recall that $Q\setminus \{1\}\subset x^{A}$ for any $x\in\Omega_{1}(Z(Q))%
\setminus \{1\}$ by $(e)$. Pick $y\in Q\cap H$ of order $q$. It holds now
that $Q\setminus \{1\}\subseteq x^{A}=y^{A}\subset H$ and the claim
follows.\medskip

$(i)$ We observe by $(g)$ that for each prime $q\in \pi(G)\setminus \{p\}$
there exists an element of order $pq$ in $G$. On the other hand the
existence of an element in $G$ of order $qr$ for distinct primes $q$ and $r$
in $\pi(G)\setminus \{p\}$ is impossible by $(c)$.\medskip

$(j)$ Notice that, for any $x\in GA$, $z^{x}$ is a $p$-element of $G$
centralizing a $q$-element of $G$ for some prime $q\ne p$ by $(g)$. Then $%
z^{x}$ is $A$-conjugate to $z$ and so $(z^{x})^{A}=z^{A}$.\medskip

$(k)$ Suppose that there exists a Sylow $p$-subgroup $P$ such that $z\in
Z(P) $. Then for any $g\in G$ there exists some $a\in A$ by $(j)$ such that $%
z^{a}\in Z(P^{g})$ which means that $z^{A}\cap Z(S)\neq\emptyset$ for any $%
S\in Syl_{p}(G)$. We shall observe that $z^{A}\sim x^{A}$ for any $1\ne x\in
G.$ This is clear by $(c)$ if $x$ is of composite order, and by $(g)$ if $%
|x| $ is a power of a prime different from $p.$ Assume now that $x$ is a $p$%
-element, and let $S\in Syl_{p}(G)$ such that $x\in S.$ As $z^{A}\cap
Z(S)\neq\emptyset$ we see that $x^{A}$ is adjacent to $z^{A}$ as claimed. As 
$x$ is arbitrary we see that $z^{A}$ is a complete vertex. It follows now by
Theorem 3.5 in \cite{IG} that $z^{A}\subset O_{p}(G)$.\medskip

$(l)-(m)$ Let the induced graph on $\{x_{i}^{A} : i=0,1,\ldots,n\}$ be a
path with $n>4$ so that $x_{0}=z$ and $x_{i}^{A}\sim x_{i+1}^{A}$ for $%
i=0,1,\ldots ,n-1.$ Without loss of generality we may assume that $%
[x_{i},x_{i+1} ]=1, i=0,1,\ldots,n-1$. By $(f)$ $|x_i|$ divide $p^2$ for
each $i>0$. Let $T\in Syl_{p}(G)$ containing $\left\langle x_{1},
x_{2}\right\rangle $ and pick a nonidentity element $t_{1}$ from $Z(T).$ We
see by $(b)$ that $t_{1}^{A}\in\{x_{1}^{A}, x_{2}^{A}\}$. On the other hand
there exists $g\in G$ such that $z^{g}\in T$. Note that $(z^{g})^{A}=z^{A}$
by $(j)$. If $t_{1}^{A}=$ $x_{2}^{A}$ then $x_{2}^{A}$ is adjacent to $z^{A}$
and hence $n=2$ and $\Gamma[\{z^{A},x_{1}^{A}, x_{2}^{A}\}]$ is $C_{3}$
which is impossible. If $t_{1}^{A}= $ $x_{1}^{A}$ and $n>2$ then we apply
the same argument to $\{x_{2}^{A}$, $x_{3}^{A}\}$ and find an element $t_{2}$
such that $t_{2}^{A}$ is adjacent to each member of $\{z^{A},x_{2}^{A}$, $%
x_{3}^{A}\}$. As $x_{2}^{A}$ is not adjacent to $z^{A}$ we see that $%
t_{2}^{A}\notin \{ z^{A},x_{2}^{A} \}$, that is, $t_{2}^{A}=x_{3}^{A}\sim
z^{A}.$ This proves that the distance of any pendant vertex from $z^{A}$ is
at most two and any cycle appearing as a subgraph of $\Gamma$ is of length
at most four.\medskip
\end{proof}

We also frequently appeal to the next proposition in the rest of this paper.

\begin{propositionn}
\label{Propo 3} Let $M$ and $N$ be $A$-invariant subgroups of $G$ such that $%
N\vartriangleleft M$. Then the following hold.

\begin{enumerate}
\item[(a)] $M/N$ is nonabelian simple if $M/N$ is a nonsolvable $A$-chief
factor of $G.$

\item[(b)] $z^{A}\cap M\subset N$ if $M/N$ is nonabelian simple.
\end{enumerate}
\end{propositionn}

\begin{proof}
$(a)$ If $M/N$ is a nonsolvable $A$-chief factor of $G$ then $M/N$ is a
direct product of isomorphic nonabelian simple groups. In case where $M/N$
is not simple there exist two prime numbers $r$ and $s$ distinct from $p$ so
that $M/N$ and hence $G$ contains an element of order $rs$ which is
impossible by Proposition \ref{Propo 2} $(c)$.\medskip

$(b)$ Let $M/N$ be a nonabelian simple group such that $z^{A}\cap
M\nsubseteq N.$ Then $z^{A}\subseteq M$ and $zN\neq N$. Let $p\neq
q\in\pi(M/N).$ Then $q\notin \pi (N)$ by Proposition \ref{Propo 2} $(h)$. On
the other hand, by Proposition \ref{Propo 2} $(g)$, for any $x\in M$ of
order $q$ there exists $a\in A$ such that $[z^{a},x]=1$ which means that $%
N\neq xN$ commutes with $z^{a}N=(zN)^{a}.$ This implies that the Gr\"{u}%
nberg-Kegel graph of the simple group $M/N$ has the vertex $p$ as a complete
vertex.

In case where $p=2$ it follows by Theorem 7.1 in \cite{Va} that $M/N$ is an
alternating group $A_{n}$ for some $n$ such that there are no prime numbers $%
r$ with $n-3\leqq r\leqq n.$ Let now $m=n$ if $n$ is odd \ and $m=n-1$ if $n 
$ is even. Then there exists an $m$-cycle, say $\sigma$, in $A_{n}.$ Now $%
|\sigma|=m$ is an odd integer which is not a prime and cannot be divisible
by two distinct primes by Proposition \ref{Propo 2} $(c)$. So $m=r^{k}$ is a
prime power with $k>1$ which is impossible by Proposition \ref{Propo 2} $(g)$
as $C_{A_n}(\sigma)=\langle\sigma \rangle$. This shows that $p$ must be odd,
and so Sylow $2$-subgroups of $M/N$ are elementary abelian by Proposition %
\ref{Propo 2} $(g)$. Using the main result of \cite{JW} we see that one of
the following holds for the simple group $M/N$:

\begin{enumerate}
\item[(i)] $M/N\cong PSL(2,2^{t})$ with $t\geq 2$;

\item[(ii)] $M/N\cong PSL(2,s)$ where $s\equiv\pm 3(mod~~8)$;

\item[(iii)] $M/N$ contains an involution $u$ with $C_{M/N} (u)=\left\langle
u\right\rangle \times K$ with $K\cong PSL(2,s)$ where $s\equiv \pm 3$ $%
(mod~~8).$
\end{enumerate}

Since $\pi (K)$ contains at least three distinct primes the case (iii)
cannot occur because otherwise we would get $p=2$. The case (i) cannot also
occur because otherwise the centralizers of involutions are Sylow $2$%
-subgroups of $M/N$ but their orders must be divisible by $p.$ Therefore we
are left with the case (ii), that is, $M/N\cong PSL(2,s)$ where $s\equiv \pm
3(mod~~8)$. Suppose that $s=3+8k$ for some $k$. Then $\left\vert
M/N\right\vert =s(1+4k)4(1+2k)$ and $M/N$ has cyclic subgroups of orders $%
(1+4k)$ and $2(1+2k)$ which contain the centralizer in $M/N$ of any of its
nontrivial elements of odd order by Kapital II Satz 8.3 and Satz 8.4 in \cite%
{Hu}. So we see that $p$ divides both $1+2k$ and $1+4k\ $which is not
possible. Similarly if $s=5+8k$ for some $k$ then we have $\left\vert
M/N\right\vert =s(3+4k)4(1+2k)$ and $M/N$ has cyclic subgroups of orders $%
(3+4k)$ and $2(1+2k)$ which contain the centralizer in $M/N$ of any of its
nontrivial elements of odd order by Kapital II Satz 8.3 and Satz 8.4 in \cite%
{Hu}. This forces that $p$ divides both $1+2k$ and $3+4k$ which is also
impossible completing the proof of (iii).
\end{proof}

\begin{propositionn}
\label{Prop 1}If a Sylow $2$-subgroup of $G$ is nonabelian then $p=2$ and
all Sylow subgroups for odd primes have prime exponent . If a Sylow $2$%
-subgroup of $G$ is abelian then either $G$ is solvable or $G$ has only one
nonabelian $A$-chief factor in an $A$-chief series and that is isomorphic to $%
PSL(2,5).$
\end{propositionn}

\begin{proof}
The first claim is clear by Proposition \ref{Propo 2} $(g)$. Assume now that $G$ is
nonsolvable and a Sylow $2$-subgroup of $G$ is abelian. Let $M/N$ be a
nonabelian $A$-chief factor of $G.$ Then by Proposition \ref{Propo 3} (b) and Proposition \ref{Propo
2} (d) $M/N$ is a simple Eppo-group with abelian Sylow $2$-subgroup and
hence isomorphic to $PSL(2,5)$ or $PSL(2,8).$

Assume first that $M/N$ is isomorphic to $PSL(2,8).$ Then $M/N$ has a cyclic
subgroup of order $9$ and so $p=3$. This shows that $2$ does not
divide $\left\vert N\right\vert \left\vert G/M\right\vert $ whence 
$M/N$ is the only nonsolvable $A$-chief factor. Therefore $N$ is
solvable. Let $T$ be a Sylow $2$-subgroup of $M.$ Then $T$ is elementary
abelian and acts by automorphisms on $N/N^{\prime }$ and $N^{\prime }$. Thus $%
N/N^{\prime }=\left\langle C_{N}(t)N^{\prime }/N^{\prime }:1\neq t\in
T\right\rangle $ and $N^{\prime }=\left\langle C_{N^{\prime }}(t):1\neq t\in
T\right\rangle $. Since it is not possible that both $N^{\prime }$ and $%
N\setminus N^{\prime }$ contain elements of $z^{A}$ we see that either $N=1$ or $N$
is an elementary abelian $p$-group generated by $z^{A}.$

If $N=1$ then $M$
is a minimal normal subgroup of $GA$ and $z\in G\setminus M$ by Proposition \ref{Propo 3} (b) and
centralizes a Sylow $p$-subgroup $\left\langle x\right\rangle $ of $M.$ But
this is not possible because then $x$ is an element of order $p^{2}$ and $%
\Gamma \lbrack z^{A},x^{A},(x^{p})^{A},(zx)^{A}]=K_{4}.$ 

If $N\neq1$ and $P\in Syl_{p}(G)$ then $N\leq P$ and hence $N\cap Z(P)\neq1.$
Without loss of generality we may assume that $z\in N\cap Z(P).$ This shows
that $\left\vert C_{M}(z)\right\vert $ is divisible by $2\cdot3^{2}\cdot
7\cdot\left\vert N\right\vert .$ As $M/N$ cannot have a proper subgroup of
index less than $5$ we see that $C_{M}(z)=M$ and hence $N=Z(M).$ As
$M^{\prime}N=M$ we have either $N\cap M^{\prime}=1$ or $N\leq
M^{\prime}.$ The first case brings us back to the situation $N=1$ which was
seen as leading to a contradiction. The second case is also not possible
because it implies that the Schur multiplier of $PSL(2,8)$ contains a
nontrivial $p$-subgroup, but it is known to be trivial.

So we have $M/N$ is isomorphic to $PSL(2,5).$ If $p\neq2$ then $\left\vert
N\right\vert \left\vert G/M\right\vert $ is odd and the claim follows. Thus we
can also assume that $p=2.$

Let us now consider the $A$-chief factor $K/S$ where $S$ is the solvable
radical of $G$. What we have seen so far shows that $K/S\cong PSL(2,5).$
$G/KC_{G}(K/S)$ is isomorphic to a subgroup of $Out(PSL(2,5)\cong Z_{2}$. To
show that $G/K$ is solvable we only need to get that $KC_{G}(K/S)/K\cong
C_{G}(K/S)/(K\cap C_{G}(K/S))=$ $C_{G}(K/S)/S$ is solvable. As for $r\in$
$\pi(K/S)\backslash\{p\}=\{3,5\}$ we know that $r$ does not divide $\left\vert
S\right\vert $, a Sylow $r$-subgroup $R$ of $K$ is a Sylow $r$-subgroup of
\ $G$ and is isomorphic to a Sylow $r$-subgroup of $K/S.$Therefore
$C_{G}(K/S)\leq C_{G}(RS/S)\leq SC_{G}(R)$ for any $r\in$ $\pi(K/S)\backslash
\{p\}$ and hence $C_{G}(K/S)\subseteq\bigcup\nolimits_{a\in A}Sz^{a}.$ This
establishes the claim.
\end{proof}

\subsection{\textbf{SOME CRITICAL GROUPS $G$ FOR WHICH $\Gamma $ IS NOT AN $%
\mathcal{F}$-GRAPH}}

\begin{lemman}
\label{Lemma1}Let $G$ be a group having a characteristic subgroup $N$ such
that $G/N\cong Sz(8).$ Then there exists no $A\leq Aut(G)$ such that $\Gamma 
$ is an $\mathcal{F}$-graph.
\end{lemman}

\begin{proof}
For\ $A\leq Aut(G)$ suppose that $\Gamma $ is an $\mathcal{F}$-graph. It is
clear by Proposition \ref{Propo 1} that $\Gamma $ is an $\mathcal{F}$-graph
with a singular vertex. As $G/N\cong Sz(8)$ holds there exists a $2$-element 
$x$ of $G$ such that $\alpha =xN$ is of order $4$. Note that there are two
conjugacy classes of elements of order $4$ represented by $\alpha $ and $%
\alpha ^{-1}$ in $Sz(8)$ and that they cannot fuse in $Aut(Sz(8))$ as $%
|Aut(Sz(8)):Inn(Sz(8))|=3$. Therefore there cannot exist any $a\in A$ such
that $x^{a}=x^{-1},$ that is, $x^{A}\neq (x^{-1})^{A}$. This yields that $%
\{x^{A},(x^{-1})^{A},(x^{2})^{A}\}$ forms a triangle and hence must contain
the singular vertex. By $(b)$ of Proposition \ref{Propo 3} we know that the
singular vertex is contained in $N.$ This contradicts the fact that $%
x^{A}\cup (x^{-1})\cup (x^{2})^{A}\subseteq G\setminus N$ and completes the
proof.
\end{proof}

\begin{lemman}
There exists no $A\leq Aut(SL(2,9))$ such that $\Gamma (SL(2,9),A)$ is an $%
\mathcal{F}$-graph.
\end{lemman}

\begin{proof}
Let $G=SL(2,9)$ and suppose that $\Gamma $ is an $\mathcal{F}$-graph for
some $A\leq Aut(G)$. Let $T\in Syl_{3}(G)$. Then $T$ is elementary abelian
of order $9,$ $C_{G}(T)=T\times Z(G)$ and $N_{G}(T)/C_{G}(T)$ is cyclic of
order $4$. Hence there are two conjugacy classes of elements of order $3$ in 
$G$. Now $Aut(G)=Inn(G)\left\langle \gamma \right\rangle $ where $\gamma $
is the automorphism of $SL(2,9)$ arising from the automorphism $x\mapsto
x^{3}$ of the field $GF(9).$ Since $\gamma $ fixes a subgroup of $G$ which
is isomorphic to $SL(2,3)$ and hence fixes an element of order $3$, we see
that the two conjugacy classes of elements of order $3$ do not fuse to one $%
Aut(G)$-orbit. If $X$ and $Y$ are representatives of different $A$-orbits of
elements of order $3$ lying in the same Sylow $3$-subgroup and $Z$ is the
involution in the center of $G$ we see that the induced graph on the set of $%
A$-orbits represented by $X,Y,XZ,YZ,Z$ is a clique which shows that $\Gamma $
is not an $\mathcal{F}$-graph. This contradiction completes the proof.
\end{proof}

\begin{lemman}
There exists no $A\leq Aut(SL(2,7))$ such that $\Gamma (SL(2,7),A)$ is an $%
\mathcal{F}$-graph.
\end{lemman}

\begin{proof}
Let $G=SL(2,7)$ and suppose that $\Gamma $ is an $\mathcal{F}$-graph for
some $A\leq Aut(G)$. A Sylow $2$-subgroup $S$ of $G$ is generalized
quaternion of order $16$ and contains a unique cyclic subgroup $%
T=\left\langle t\right\rangle $ of order $8$. Notice that any two elements
of order $8$ in $T$ are conjugate to each other in $H=Aut(G)$ if and only if
they are conjugate by an element in $N_{H}(T)=SC_{H}(T).$ It follows that
the induced graph on the set of $A$-orbits represented by $%
t,t^{2},t^{3},t^{4}$ is $K_{4}$, which is impossible. This proves the claim.
\end{proof}

\subsection{\textbf{THE CASE WHERE $O_{p}(G)\neq 1$}}

\begin{propositionn}
\label{Propo 2**} If $O_{p}(G)\neq1$ then $z\in O_{p}(G).$
\end{propositionn}

\begin{proof}
Suppose that $O_{p}(G)\neq 1$ and $z\notin O_{p}(G).$ Set $\overline{G}%
=G/O_{p}(G).$ We observe by Proposition \ref{Propo 2} $(c)$ that $%
C_{G}(O_{p}(G))\leq O_{p}(G)$.

Let $\overline{M}=M/O_{p}(G)$ be a minimal $A$-invariant normal subgroup of $%
\overline{G}.$ Assume first that $M\leq O_{p,p^{\prime }}(G)$. As $z\notin M$
the group $M$ is a Frobenius group with kernel $O_{p}(G)$ by Proposition \ref%
{Propo 2} $(c)$. It follows that $\overline{M}$ is elementary abelian and
hence is of order $r$ for some prime $r\neq p.$ Let $R$ be a Frobenius
complement of $O_{p}(G)$ in $M.$ Then $G=O_{p}(G)N_{G}(R).$ Note that an $A$%
-conjugate of $z$ centralizes $R$ by Proposition \ref{Propo 2} $(g)$. It
follows by Proposition \ref{Propo 2} $(c)$ that $C_{G}(R)=RS$ where $S$ is a
nontrivial $p$-group. Then $O_{p}(O_{p}(G)C_{G}(R))=O_{p}(O_{p}(G)(R\times
S))=O_{p}(G)S$ which leads to the contradiction that $S=1.$ Thus $\overline{M%
}$ is not a $p^{\prime }$-group. On the other hand $\overline{M}$ is not a $%
p $-group and hence is a direct product of isomorphic nonabelian simple
groups. As in the proof of Proposition \ref{Propo 3} $(a)$ we conclude that $%
\overline{M}$ is a nonabelian simple group. Then $z\notin M$ by Proposition %
\ref{Propo 3} $(b)$ and hence $M$ is a nonsolvable Eppo-group by Proposition %
\ref{Propo 2} $(d)$. As $M$ has a nontrivial normal $p$-subgroup, Theorem %
\ref{Theo 1} implies that $p=2$ and $\overline{M}$ is isomorphic to either $%
PSL(2,q)\;\text{for}\;q\in \{4,8\}\;\text{or}\;Sz(8)\;\text{or}\;Sz(32).$

Notice that the outer automorphism groups of $PSL(2,8)$, $Sz(8)$ and $Sz(32) 
$ are of odd order. Therefore $z$ induces an inner automorphism on $%
\overline{M}$ in each of these cases, that is, there exists $x\in M$ with $%
[M,xz]\leq O_{p}(G)$. By Proposition \ref{Propo 2} $(c)$ it holds that $%
xz\in z^{A}$. On the other hand, by a similar argument as above, one can see
that $O_{p}(G)\left\langle xz\right\rangle $ is a normal $p$-subgroup of $G$%
. Then $xz\in z^{A}\cap O_{p}(G)$ which is not the case. So we are left with
the case $\overline{M}\cong PSL(2,4)\cong A_{5}$ whence $\overline{G}\cong
S_{5}$ which is also impossible since no element of order $5$ is centralized
by an involution in $S_5$. This completes the proof.
\end{proof}

In this section we study the case where $O_{p}(G)\ne 1$ by examining the
subcases $O_{p}(G)\neq F(G)$ and $O_{p}(G)=F(G)$ separately. The next result
provides a precise description of the structure of $G$ when $O_{p}(G)$ and $%
O_{p^{\prime}}(F(G))$ are both nontrivial.

\begin{propositionn}
\label{Propo 6} If $1\neq O_{p}(G)\neq F(G)$ then $F(G)=P\times Q$ where $%
P=O_{p}(G)$ is an elementary abelian $p$-group, $Q=O_{p^{\prime}}(F(G))$ is
a Sylow $q$-subgroup of $G$ for another prime $q$ which is elementary
abelian, $G/P$ is a Frobenius group with kernel $F(G)/P$ and complement
either of prime order or a $p$-group which has a unique subgroup of order $p 
$. Furthermore both $P\setminus \{1\}$ and $Q\setminus \{1\}$ are $A$-orbits.
\end{propositionn}

\begin{proof}
By hypothesis $Q\ne 1$. Since $Q$ is an $A$-invariant, normal and nilpotent
subgroup it follows by $(c),(g)$ and $(h)$ of Proposition \ref{Propo 2} that 
$Q\ $is an elementary abelian $q$-group for some prime $q$ and is a Sylow $q$%
-subgroup of $G$. Thus we have $F(G)=Q\times P.$ Note that $\Gamma\ne C_{3}$
and hence $F(G)\ne G$. We observe that $z^{A}\subset P$ by Proposition \ref%
{Propo 2**}. Then $C_{F(G)/P}(xP)=1$ for any $x\in G\setminus F(G)$ by
Proposition \ref{Propo 2} $(c)$, that is, $G/P$ is a Frobenius group with
kernel $F(G)/P$. This forces that $\left\vert \pi(G/F(G))\right\vert =1$,
and if $\pi(G/F(G))\neq\{p\}$ then $G/F(G)$ is cyclic of prime order as
claimed.
\end{proof}

\begin{example}
Let $G=Q\rtimes T$ where $Q=\left\langle \sigma\right\rangle $ is of order $%
3 $ and $T=\left\langle \tau\right\rangle $ is of order $4$ and $%
\sigma^{\tau }=\sigma^{-1}.$ Then $F(G)=P\times Q$ where $P=\left\langle
\tau^{2}\right\rangle$, and $G/O_{2}(G)$ is the Frobenius group of order $6$%
. There exists an automorphism $\alpha$ of $G$ such that $\alpha$ fixes $%
\sigma$ and inverts $\tau.$ Let $t$ denote the inner automorphism of $G$
induced by $\tau$ and $A=\left\langle t,\alpha\right\rangle$. Then $%
\Gamma(G,A)$ is an $\mathcal{F} $-graph with $5$ vertices, one triangle $%
\Gamma[({\tau}^2)^A, \sigma^A, ({\tau}^2\sigma)^A]$ and two $P_2$'s, namely $%
\Gamma[\tau^A, ({\tau}^2)^A]$ and $\Gamma [({\tau}^2)^A, (\tau\sigma)^A]$.
\bigskip
\end{example}

The rest of this subsection is devoted to the case where $1\neq
O_{p}(G)=F(G) $. First we handle the subcase where $C_{G}(F(G))\leq F(G)$.

\begin{propositionn}
\label{Propo 7} Suppose that $1\neq O_{p}(G)=F(G)$ and that $C_{G}(F(G))\leq
F(G)$ holds. Set $\overline{G}=G/F(G).$ Then one of the following holds.

\begin{enumerate}
\item[(a)] $\overline{G}$ has a normal subgroup $\overline{M}$ of prime
order which is complemented in $\overline{G}$ by a $p$-group of order
dividing $p$ and acting Frobeniusly on it.

\item[(b)] $\overline{G}\cong SL(2,4)$,\thinspace\ $p=2$,\thinspace\ $%
Z(G)=z^{A}\cup \{1\},\;1\neq F(G)/Z(G)$ is a direct sum of natural
irreducible modules of $\overline{G}.$ Furthermore $G$ does not split over $%
Z(G)$.

\item[(c)] $\overline{G}\cong SL(2,4),\;p=5,\;F(G)$ is elementary abelian, $%
F(G)\setminus \{1\}$ is the union of two $A$-orbits, all $p$-elements in $%
G\setminus F(G)$ lie in the same $A$-orbit, $z^{A}\cap Z(S)=\phi $ for any $%
S\in Syl_{p}(G)$, and $A/C_{A}(\overline{G})$ is isomorphic to $S_{4}$ or $%
S_{5}$.
\end{enumerate}
\end{propositionn}

\begin{proof}
Let $P=F(G)$. By Proposition \ref{Propo 2**} we know that $z\in P$ and $%
\overline{G}$ is an Eppo-group.\newline

\textit{Step 1. If }$G$\textit{\ is solvable then }$(a)$\textit{\ holds.}%
\newline

\textit{Proof.} Suppose that $G$ is solvable. Then $\overline{G}$, being a
solvable Eppo-group, is either a $q$-group for some prime $q\neq p$ or is a
Frobenius or a $2$-Frobenius group with $\left\vert \pi (\overline{G}%
)\right\vert =2.$ Let $\overline{M}$ be a minimal normal $A$-invariant
subgroup of $\overline{G}.$ It is an elementary abelian $q$-group for some
prime $q\neq p$ and is a Sylow $q$-subgroup of $\overline{G}$ by Proposition \ref{Propo
2} $(h).$ This shows that $\overline{G}$ cannot be $2$-Frobenius with $%
\left\vert \pi (\overline{G})\right\vert =2$ and hence we have either $%
\overline{G}$ is a $q$-group or $\overline{G}$ is a Frobenius group with
kernel $\overline{M}$ whose complement is an $r$-group for some prime $r.$
If $r\neq p$ then $|G/M|=r$ by Proposition \ref{Propo 2} (g). If $r=p$ then $G/M$ is
isomorphic to a $p$-group which has a unique subgroup of order $p$. For any $
Q\in Syl_{q}(G)$ and $R\in Syl_{r}(N_{G}(Q))$ we have $M=PQ$ and $G=MR.$

Suppose that either $Q$ is noncyclic or $r\neq p.$ Then $N=\left\langle
C_{N}(x):1\neq x\in Q\right\rangle $ or $C_{[N,Q]}(R)\neq 1$ for any $QR$-invariant section $N$ of $P.$ This shows that $\Phi (P)=1$ because
otherwise both $\Phi (P)$ and $P\setminus \Phi (P)$ contain $A$-conjugates
of $z$ which is impossible. So $P$ is an elementary abelian $p$-group. Let $
P=P_{1}\oplus P_{2}\oplus \cdots \oplus P_{k}$ be the decomposition of $P$
into the sum of homogeneous $Q$-components. For any $x=x_{1}+\cdots
+x_{k}\in P$ with $x_{i}\in P_{i},\,i=1,2,\ldots ,k$, we define the weight
of $x$ as $\left\vert \{i:x_{i}\neq 1\}\right\vert $. As $A$ acts on the set 
$\{P_{1},\ldots ,P_{k}\}$ we see that $A$ stabilizes the sets of elements of 
$P$ of the same weight.\ Therefore $%
v_{1}^{A},(v_{1}+v_{2})^{A},(v_{1}+v_{2}+v_{3})^{A}$ where $1\neq v_{i}\in
P_{i}$ for $i=1,2,\ldots ,k$, are different $A$-orbits. If $Q$ is noncyclic then $%
C_{Q}(P_{i})$ is nontrivial for each $i=1,2,\ldots ,k$. It follows that $%
k=1\,\ $because otherwise $z^{A}=v_{1}^{A}\neq (v_{1}+v_{2})^{A}=z^{A}$
where $1\ne v_{i}\in C_{Q}(P_{i}),i=1,2.$ Then $P_{1}=P$ and $%
C_{Q}(P)\neq 1$ which is not possible. Thus $Q$ is cyclic. Clearly $C_{G}(z)$
has order divisible by $q$ and $r$ if $r\neq p$ and hence $%
C_{G}(z)=G$. Now $z^{A}\subset Z(G)<P$ as $C_{G}(P)\leq P.$ But then $%
z\notin N=[P,Q]\neq 1$. As $QR$ is a Frobenius group we obtain that $%
C_{N}(R)\neq 1$. This yields a contradiction if $r\neq p$ because it implies
that $N$ contains an $A$-conjugate of $z$ which is not possible. Thus $r=p$.
Since $Q$ is cyclic we see that $R$ is a cyclic $p$-group. If possible, let $%
x$ be a $p$-element such that $p^{2}$ divides the order of $xP$. If $\
x^{A},(x^{p})^{A},(x^{-1})^{A}$ are pairwise different $A$-orbits then they
are adjacent to each other and hence one of them must be $z^{A}$ which is
not possible. So there must exist an element $a\in A$ such that $%
x^{a}=x^{-1}$. But then the subgroup of the semidirect product $GA$
generated by $x$ and $a$ must induce a group $L$ of automorphisms on the
cyclic section $\overline{M}$ which must be isomorphic to a cyclic group
of order dividing $q-1$. Of course this is not possible as $L$ is not abelian. Therefore $|\overline{G}:$ $\overline{M}|$ divides $p.$\\

\textit{Step 2. If }$G$\textit{\ is nonsolvable we have }$O_{2}(\overline{G}%
)=1$\textit{\ and hence }$\overline{G}$ \textit{is either a simple
Eppo-group or isomorphic to }$M_{10}$.\newline

\textit{Proof. }If $G$ is nonsolvable the assumption $O_{2}(\overline{G}
)\neq1\ $leads to $p\neq2$ as $P=O_{p}(G)$. Therefore a Sylow $2$-subgroup
of $G$ is of exponent $2$ and hence is elementary abelian. This implies that
a Sylow $2$-subgroup of $\overline{G}/O_{2}(\overline{G})$ acts trivially on 
$O_{2}(\overline{G})$ which is not possible. Then $O_{2}(\overline{G})=1$ and
hence $\overline{G}$ is either a simple Eppo-group or isomorphic to $M_{10} $%
.\newline

\textit{Step 3. If }$G$\textit{\ is nonsolvable and }$C_{G}(z)$\textit{\
contains a Sylow }$p$\textit{-subgroup of }$G$\textit{\ then} $(b)$ 
\textit{holds.}\newline

\textit{Proof. } We shall first observe that $z\in Z(G)$ under the
assumption that $C_{G}(z)$ contains a Sylow $p$-subgroup of $G$: For any
prime $q\in \pi (G)\setminus \{p\}\subset \pi (\overline{G}),$ a
Sylow $q$-subgroup of $G$ is a group of exponent $q$ which is isomorphic to
a Sylow $q$-subgroup of $\overline{G}$ where $q$ divides $|C_{G}(z)|.$

Assume now that the Sylow $2$-subgroups of $\overline{G}$ are not elementary
abelian. Then $p=2$ and by Theorem \ref{Theo 1}, the group $\overline{G}$ is
isomorphic to one of the following groups:%
\begin{equation*}
PSL(2,7),PSL(2,9),PSL(2,17),PSL(3,4),Sz(8),Sz(32),M_{10}
\end{equation*}%
We can eliminate the groups $PSL(2,17)$ and $Sz(32)$ from this list because
they contain cyclic groups of orders $9$ and $25$, respectively. Among the
remaining groups, $PSL(2,9),$ $M_{10}$ and $PSL(3,4)$ have elementary
abelian Sylow $3$-subgroups of order $9$ and cyclic Sylow subgroups for
primes different from $2$ and $3$. So $|G:C_{G}(z)|$ divides $3$ and hence
is $1.$ The others, namely $PSL(2,7)$ and $Sz(8)$, have cyclic Sylow
subgroups for odd primes. Therefore $G=C_{G}(z)$ in case
where the Sylow $2$-subgroups of $\overline{G}$ are not elementary abelian.

Assume next that the Sylow $2$-subgroups of $\overline{G}$ are abelian. Then 
$\overline{G}\cong PSL(2,5)$ by Proposition \ref{Prop 1}. We get $%
[G:C_{G}(z)]\leq 4$ and hence $z\in Z(G)$ establishing the first claim.

As $z\in Z(G)$, we have $\left\langle z^{A}\right\rangle \leq Z(G)=z^{A}\cup
\{1\}.$ In this case $Z(G)$ is properly contained in $P$ since $C_{G}(P)\leq
P$. Hence, by Proposition \ref{Propo 2} $(e)$, $G/Z(G)$ is a nonsolvable
Eppo-group with nontrivial normal subgroup $P/Z(G)$. In particular $p=2$ and 
$\overline{G}$ is isomorphic to one of the groups $%
SL(2,4),Sz(8),SL(2,8),Sz(32)$. The last two can be eliminated from this list
because they contain elements of orders $9$ and $25$ respectively. $Sz(8)$
cannot also occur by Lemma \ref{Lemma1}.

To complete the proof of Step 3 it remains only to show that $G$ does not
split over $Z(G)$. Assume the contrary, that is, assume that $G=H\times Z(G)$
for some subgroup $H$ which is isomorphic to $G/Z(G)$. Since $G/Z(G)$ is a
perfect group we see that $H=H^{\prime }=G^{\prime }$ is $A$-invariant. As $%
H $ contains an element $h$ of order $4$, the induced graph on
the set of $A$-orbits represented by $h,h^{2},hz,h^{2}z$ is $K_{4}$. This
contradiction completes the proof. \newline

\textit{Step 4. If }$G$\textit{\ is nonsolvable and }$C_{G}(z)$ \textit{does
not contain a Sylow }$p$\textit{-subgroup of }$G$\textit{, then }$(c)$%
\textit{\ holds.}\newline

\textit{Proof. } Let $N$ be an $A$-invariant minimal normal subgroup of $G$
contained in $P$. Clearly $N\leq Z(P)$. Let $S\in Syl_{p}(G)$ and $1\neq
x\in N\cap Z(S).$ Then we have $x^{A}\neq z^{A}.$ Assume that $A$ has $k$
orbits in $N\setminus \{1\}$. Note that $N<S$ because otherwise $z\in Z(S)$.
Then for any $y\in S\setminus N$ the orbit $y^{A}$ is adjacent to $x^{A}$
and is different from the vertices contained in the set $N.$ As $\deg
(x^{A})\leq 2$ we see that $k\leq 2.$ If $k=1,$ then $z^{A}\cap N=\emptyset $
and hence $N<P$ as $z\in P$. On the other hand, $\overline{G}$ acts on $N$
in such a way that every $p^{\prime }$-element of $\overline{G}$ is fixed
point free on $N$. Then $E=N\rtimes \overline{G}$ is a nonsolvable
Eppo-group with nontrivial normal subgroup $N$. We see by Theorem \ref{Theo
1} that $p=2.$ As $\overline{G}$ is nonsolvable, $|\overline{G}|$ is even
and hence $P<S$. Again the fact that $\deg (x^{A})\leq 2$ yields $P\setminus
N=z^{A}$ and $S\setminus P\subset y^{A}$ for any $y\in S\setminus P.$ This
implies in particular that the exponent of $\overline{S}$ of $\overline{G}$\
must be $2$ and hence $\overline{S}$ is abelian. Thus we have either $%
\overline{G}\cong PSL(2,5)$ or $\overline{G}\cong PSL(2,8).$ The latter
cannot occur since $PSL(2,8)$ contains elements of order $9$ and $p=2.$ On
the other hand as $P\setminus N=z^{A}$ we see that $\exp (P)=2$ which
implies that $P\leq C_{G}(z)$. Hence $|G:C_{G}(z)|\leq 4,$ as $3$ and $5$
have to divide $|C_{G}(z)|.$ But the simple group $\overline{G}$ cannot have
a proper subgroup of index less than $5$ and so $z\in Z(G)$ in case where $%
k=1$. Therefore we must have $k=2$, that is, $N\setminus \{1\}$ is the union
of two $A$-orbits. \ \ 

If $N<P$ then $S=P$ since $\deg (x^{A})\leq 2$. This means $p$ does not divide $%
\left\vert \overline{G}\right\vert $ which shows in particular that $p$ is
odd. Hence a Sylow $2$-subgroup of $G$ is noncyclic, elementary abelian.
If $T\in Syl_{2}(G)$ then we have  $P/N=\left\langle C_{P/N}(x):1\neq x\in
T\right\rangle $ where $N=\left\langle C_{N}(x):1\neq x\in T\right\rangle $
which implies that ($P\setminus N)\cap z^{A}\neq \emptyset \neq N\cap z^{A}$%
, a contradiction. Thus $P=N$ and hence $P$ is abelian. As $z\notin Z(S)$,
we see that $P<S.$ We also clearly have $P\setminus \{1\}=x^{A}\cup z^{A}$
for some $x\in P\cap Z(S)$.

If Sylow $r$-subgroups of $\overline{G}$ are not of exponent $r$ then $p=r$
and $S\setminus P$ intersects at least two $A$-orbits, corresponding to
elements of order $p$ and $p^{2}$ nontrivially, which is not the case as $%
\deg(x^{A})\leq2$. So every Sylow subgroup of $\overline{G}$ is of prime
exponent, in particular Sylow $2$-subgroups of $\overline{G}$ are elementary
abelian and hence $\overline{G}$ is isomorphic to either $PSL(2,5)$ or $%
PSL(2,8)$. The latter case cannot occur because a Sylow $3$-subgroup of $%
PSL(2,8)$ is cyclic of order $9$. If $p=2$ then $3$ and $5$ divide the order
of $C_{G}(z)$ and hence $|G:C_{G}(z)|$ divides $4$ which is impossible since 
$PSL(2,5)$ has no proper subgroup of index less than $5$. Therefore $p$ is odd, $%
\overline{G}\cong PSL(2,5)$, and $|G:C_{G}(z)|$ divides $2p$ and is
divisible by $p$.

Let $yP$ be a nonidentity element in $\overline{G}$ where $y$ is a $p$%
-element in $G\setminus P,$ and let $S_{1}\in Syl_{p}(G)$ with $y\in S_{1}$. Then
clearly $P\trianglelefteq S_{1}$ and $y^{A}$ is adjacent to $v^{A}$ for any $%
v\in Z(S_{1})\cap P.$ Recall that $P\setminus \{1\}=x^{A}\cup z^{A} $ for
some $x\in P\cap Z(S)$. As $C_{G}(z)$ does not contain a Sylow $p$-subgroup
of $G$ we have $v^{A}=x^{A}$. Since $\deg(x^A)=2$ it holds that all the $p 
$-elements in $G\setminus P$ lie in the same $A $-orbit, in particular all
the $p$-elements in $G\setminus P$ are of the same order.

Let $T\in Syl_{2}(G)$. Clearly $T$ is elementary abelian of order $4$ as $p$
is odd and we have 
\begin{equation*}
\lbrack P,T]=P_{1}=C_{P_{1}}(t_{1})\oplus C_{P_{1}}(t_{2})\oplus
C_{P_{1}}(t_{3})
\end{equation*}%
where $T=\{1,t_{1},t_{2},t_{3}\}=\langle t_{1},t_{2}\rangle .$ Notice that $%
C_{P_{1}}(t_{1})$ is $T$-invariant, and $t_{2}$ acts fixed point freely on $%
C_{P_{1}}(t_{1})$ whence $t_{2}$ inverts elements of $C_{P_{1}}(t_{1})$.
Since nontrivial elements of $C_{P_{1}}(t_{1})$ are conjugate to $z$ we see
that $N_{G}(\left\langle z\right\rangle )$ contains $C_{G}(z)$ properly . So
we have $[G:N_{G}(\left\langle z\right\rangle )]=p$. It follows that $p=5$
because $G$ cannot have a subgroup of index less than $5.$

Clearly $A$ normalizes $F(G)$ and induces automorphisms on $\overline{G}.$
Let $A_{1}=C_{A}(\overline{G}).$ Then $\overline{A}=$ $A/A_{1}\leq Aut(%
\overline{G})\cong S_{5}$. We see that the involutions of $T$ must lie in
the same $A$-orbit and for that purpose it is sufficient and necessary that
there exists an element in $A/A_{1}$ that leaves $\overline{T}$ invariant
and acts on it as an automorphism of order $3.$ Observe that the normalizer
of a Sylow 2-subgroup of $A_{5}$ in $S_{5}$ is isomorphic to $S_{4}$. So the
number of Sylow 3-subgroups of $\overline{A}\leq S_{5}$ is either $1$ or $4$ or $10$. The first
cannot occur since a Sylow $3$-subgroup of $S_{5}$ normalizes only $2$ Sylow
$2$-subgroups and acts transitively on the remaining $3$ Sylow $2$-subgroups. On
the other hand we know that $\overline{A}$ is not contained in $Inn(%
\overline{G})\cong A_{5},$ because if $g$ is an element of order a power of $%
5$ such that $\overline{g}$ an element of order $5$ of $\overline{G}$ then $%
\overline{g}$ and $\overline{g}^{2}$ are not conjugate in $Inn(\overline{G})$
and hence $g^{A}$ and $(g^{2})^{A}$ are two adjacent, different vertices,
which is impossible. Therefore $\overline{A}$ is isomorphic to $S_{4}$ or $S_{5}$
establishing the claim.
\end{proof}

\begin{remark}
It is very desirable to pin down the structure of $F(G)$ in part (b) of the
above theorem and decide whether this case can really occur. The same
question about the possible nonoccurence exists also in part (c) in the
light of our knowledge of irreducible $S_{5}$-modules over $GF(5).$ These
contain more than two $S_{5}$-orbits of nontrivial elements. On the other
hand $F(G)$ is a minimal normal subgroup of $GA$ and hence is a homogeneous $%
\overline{G}$-module and also a homogeneous $C_{A}(\overline{G})$-module.
\end{remark}

.

\begin{propositionn}
\label{Propo 8} If $1\neq O_{p}(G)=F(G)$ and $C_{G}(F(G))\nleq F(G)$ then
either $G$ is quasisimple and isomorphic to one of $SL(2,5), 2PSL(3,4), 
2^{2}PSL(3,4)$ or $G=O_{p}(G)\times E(G)$ where $O_{p}(G)=\left\langle
z^{A}\right\rangle $ and $E(G)\cong PSL(2,5).$
\end{propositionn}

\begin{proof}
The generalized Fitting subgroup $F^{\ast }(G)$ is of the form $F^{\ast
}(G)=F(G)E(G)$ where $E(G)$ is the layer of $G$, that is the subgroup
generated by components, namely, subnormal quasisimple subgroups. As $%
C_{G}(F(G))\not\leq F(G)$ we see that $E(G)\neq 1$. Also note that $E(G)$
must be quasisimple since different components of $G$ centralize each other.
Clearly $Z(E(G))\subseteq z^{A}\cup \{1\}$ and hence 
\begin{equation*}
Z(E(G))=1\;\;\text{or}\;Z(E(G))\supseteq z^{A}\cup \{1\}\supseteq
\;C_{G}(E(G))\geq F(G).
\end{equation*}%
This observation shows that one of the following holds:

\begin{itemize}
\item $E(G)$ is a simple Eppo-group and $F^{\ast }(G)=\langle z^{A}\rangle
\times E(G)$,

\item $F^{\ast }(G)=E(G)$ is quasisimple with nontrivial center $Z(E(G))=$ $%
\langle z^{A}\rangle $ and $E(G)/Z(E(G))$ is a simple Eppo-group.
\end{itemize}

In any case $F^{\ast }(G)/Z(F^{\ast }(G))$ is isomorphic to one of the
following groups: 
\begin{equation*}
PSL(2,q)\;\;\text{for}\;\;q\in \{5,7,8,9,17\},\, PSL(3,4),\, Sz(8),\, Sz(32).
\end{equation*}%
In the case $Z(E(G))=1$ every element of the simple group $E(G)$ must be of
prime order, in particular $E(G)$ must have elementary abelian Sylow $2$%
-subgroups. The only simple Eppo-group satisfying these conditions is $%
PSL(2,5)$ whence we have $F^{\ast }(G)=\langle z^{A}\rangle \times PSL(2,5).$

Suppose now that $E(G)$ is quasisimple with nontrivial center and look at
the list of all possible groups for $E(G)/Z(E(G))$. The groups $PSL(2,8)$, $%
Sz(32)$ both have trivial Schur multipliers. Since the Schur multiplier of $%
PSL(2,17)$ is of order $2$ we would get $p=2$ in this case and this would
lead to a contradiction as $PSL(2,17)$ contains an element of order $9$.
Thus $F^{\ast }(G)/Z(F^{\ast }(G))$ is isomorphic to one of the following
groups: $PSL(2,5),\;PSL(2,7),\;PSL(2,9),\;PSL(3,4),\;Sz(8).$ Notice that the
Schur multiplier of $PSL(2,9)$ is of order $6$, but $p\neq 3$ because
otherwise a Sylow $2$-subgroup of $PSL(2,9)$ must be elementary abelian
which is not the case. Therefore in the first three cases we have $F^{\ast
}(G)$ is isomorphic to one of $SL(2,5),\;SL(2,7),\;SL(2,9).$ The Schur
multiplier of $PSL(3,4)$ is isomorphic to $\mathbb{Z}_{4}\times \mathbb{Z}%
_{12}$ so that $Z(F^{\ast }(G))$ is isomorphic to one of the groups $\mathbb{%
Z}_{2},\;\mathbb{Z}_{2}\times \mathbb{Z}_{2},\;\mathbb{Z}_{3}$. Again the
last one is not possible as a Sylow $2$-subgroup of $PSL(3,4)$ is not of
exponent $2$. Essentially there exist up to isomorphism one quasisimple
group $K_{1}$ such that $Z(K_{1})\cong \mathbb{Z}_{2}$ and $%
K_{1}/Z(K_{1})\cong PSL(3,4)$ and one quasimple group $K_{2}$ such that $\
Z(K_{2})\cong \mathbb{Z}_{2}\times \mathbb{Z}_{2}$ and $K_{2}/Z(K_{2})\cong
PSL(3,4)$ possessing an automorphism acting transitively on the set of the
involutions in the center. Similarly the Schur multiplier of $Sz(8)$ is
elementary abelian of order $4$ and the full covering group of $Sz(8)$ has
an automorphism of order $3$ which acts transitively on the set of
involutions of the center of this covering group. So there exist unique
quasisimple groups $L_{1}$ and $L_{2}$ with $Z(L_{1})\cong \mathbb{Z}_{2}$
and $Z(L_{2})\cong \mathbb{Z}_{2}\times \mathbb{Z}_{2}$ and $%
L_{i}/Z(L_{i})\cong Sz(8)$ for $i=1,2$.

In both cases $F(G)=\langle z^{A}\rangle $ and hence $G/F(G)$ is a
nonsolvable Eppo-group having a simple normal subgroup $F^{\ast }(G)/F(G).$
Appealing to the list of such Eppo-groups we obtain that $G=F^{\ast }(G)$
since the only nonsimple Eppo-group $M_{10}$ cannot occur as $G/F(G)$
because the covering group $2PSL(2,9)$ cannot be extended to $M_{10}$ whose
Schur multiplier is of order $3$. By the Lemmas of subsection
3.2 we can conclude that $G/Z(G)$ cannot be isomorphic to any of $%
PSL(2,7),\;PSL(2,9),\;Sz(8)$. As a result $G$ is isomorphic to one of 
\begin{equation*}
SL(2,5),\;2PSL(3,4),\;2^{2}PSL(3,4)
\end{equation*}%
establishing the claim.
\end{proof}

\begin{example}
Let $G=SL(2,5)$. Then $Aut(G)$ has $6$ orbits on $G\setminus \{1\}$ of lengths $%
1,20,20,30,24,24$, represented by elements of orders $2,3,6,4,5,10$
respectively. One can easily check that $\Gamma (G,Aut(G))$ is an $\mathcal{F%
}$-graph consisting of $2$ triangles and a tail $P_{2}$ joined at the
singular vertex corresponding to the central element. The unique vertex with
degree $1$ corresponds to the orbit consisting of the elements of order $4.$
\end{example}

\begin{example}
Let $G=Z\times H$ where $Z$ is an elementary abelian $p$-group for some
prime $p$ and $H=A_5\cong PSL(2,5).$ Take further $A\leq Aut(G)$ given as $%
A=U_{1}\times U_{2}$ where $\ U_{1}$ is cyclic of order $\left\vert
Z\right\vert -1$ and acts Frobeniusly on $Z$ fixing $H$ elementwise, and $%
U_{2}=S_5$ fixes $Z$ elementwise and acts as $Aut(H)$ on $H.$ Then 
$\Gamma =\Gamma (G,A)$ is a friendship graph with $3$ triangles.

Let $B=U_{1}\times U_{3}$ with $U_{3}=N_{U_{2}}(T)$ for $T\in
Syl_{2}({U_{2}}')$ where $U_{3}$ is considered as the subgroup of $S_{5}$ which stabilizes the point $5$
 in the natural permutation representation of $S_{5}$
on the set $\{1,2,3,4,5\}$. Then $U_3\cong S_4$ and the $U_3$-orbits on $H=A_5$ can be represented by
the elements%
\begin{equation*}
(1,2)(3,4),(1,2,3),(1,2,3,4,5),(1,2)(3,5),(1,2,5)
\end{equation*}
are respectively of lengths $3,8,24,12,12.$ So $\Gamma (H,U_{3})=\Gamma (A_{5},S_{4})$ has $5$ vertices and no edges and
that $\Gamma (G,B)$ is a friendship graph with $5$ triangles.
\end{example}

\subsection{\textbf{THE CASE WHERE $O_{p}(G)=1$}}

\begin{propositionn}
\label{Propos 2} Suppose that $O_{p}(G)=1$. Then $p=2$ and $G=F(G)S$ where $
S\in Syl_{2}(G)$ and $F(G)$ is  elementary abelian.
Moreover one of the following holds where $H=GA$ and $\overline{H}%
=H/C_{H}(F(G)):$

\begin{enumerate}
\item[(i)] $S\cong D_{8}$, $\left\vert \overline{H}/\overline{G}\right\vert
=q-1 $ and $\left\vert F(G)\right\vert =q^{2}$ where $q\in \{3,7\}$.

\item[(ii)] $S\cong Q_{8}\ast D_{8}$, the extraspecial group of order $32$
with $20$ elements of order $4$, $\left\vert F(G)\right\vert =3^{4}$, and $%
\overline{H}/\overline{G}$ is a Frobenius group of order $10$ or $20.$
\end{enumerate}
\end{propositionn}

\begin{proof}
For the sake of easy understanding we shall divide the rest of the proof
into smaller steps.\newline

\textit{Step 1. }$G$\textit{\ is a }$\{p,q\}$\textit{-group where }$F(G)$%
\textit{\ is a Sylow }$q$\textit{-subgroup of }$G$\textit{\ which is
elementary abelian, self-centralizing and }$|Z(G/F(G))|=p$. \textit{%
Furthermore nontrivial elements of} $F(G)$ \textit{constitute an} $A$-%
\textit{orbit}.\newline

\textit{Proof.} Let $M$ be an $A$-invariant minimal normal subgroup of $G.$
As $O_{p}(G)=1$ we see that $C_{G}(M)=1$ or $C_{G}(M)=\dot{M}.$ In the first case $M$ is nonsolvable. Then $M$ is nonabelian simple and hence $%
z^{A}\cap M=\emptyset $ by Proposition \ref{Propo 3} $(a)$-$(b)$. Now
Proposition \ref{Propo 2} $(d)$ implies that $M$ is a simple Eppo-group and $%
G$ is isomorphic to a subgroup of $Aut(M)$ containing $M$. Clearly $z$
induces by conjugation an automorphism $\zeta $ on $M$. Note that $\zeta $
is an outer automorphism because otherwise there would exist $x\in M$ such
that $xz\in C_{G}(M)=1$. Hence $p$ divides $\left\vert Out(M)\right\vert.$ 
If $p$ is odd then a Sylow $2$-subgroup of $M$ is abelian by Proposition \ref%
{Propo 2} $(g)$. Now appealing to Propositio \ref{Prop 1} we observe that $M$ is
isomorphic to one of $PSL(2,5).$ But $\left\vert Out(PSL(2,5)\right\vert =2.$ This shows that $p=2.$

Then all Sylow $q$-subgroups of $M$ for odd primes $q$ are of exponent $q$
by Proposition \ref{Propo 2} $(g)$. We observe by Theorem \ref{Theo 1} that $%
M$ is isomorphic to one of the following groups: 
\begin{equation*}
PSL(2,5),PSL(2,7),PSL(2,9),PSL(3,4)
\end{equation*}%
because $PSL(2,8)$ and $PSL(2,17)$ do not satisfy the exponent condition,
and both $|Out(Sz(8))|$ and $|Out(Sz(32))|$ are odd. Notice that in all
these remaining cases there exists a prime $q\neq p$ such that $q$ does not
divide $\left\vert C_{M}(\zeta )\right\vert $ which is impossible by $(g)$
of Proposition \ref{Propo 2}. Therefore $M$ must be an elementary abelian $q$
-group for some prime $q\neq p.$ We have $M\in Syl_{q}(G)$ and $M\setminus
\{1\}$ is an $A$-orbit by $(g)$ and $(h)$ of Proposition \ref{Propo 2}.
Clearly it also holds that $C_{G}(M)=M$.

We know by part $(k)$ of Proposition \ref{Propo 2} that $z^{A}\cap
Z(S)=\emptyset $ for any Sylow $p$-subgroup $S$ of $G.$ Let $S\in Syl_{p}(G)$%
. If $1\neq u\in Z(S)$ then $C_{G}(u)=S$ as $u\notin z^{A},$ in particular $%
u $ acts fixed point freely on $O_{p^{\prime }}(G)$ for any $1\neq u\in Z(S)$
$.$ This shows first that $O_{p^{\prime }}(G)$ is nilpotent and hence $%
O_{p^{\prime }}(G)=M$ and second that $Z(P)$ is cyclic$.$ Suppose that there
exists an element $x$ in $Z(S)$ of order $p^{2}$. Then for any $y\in
S\setminus Z(S)$ we see that $y^{A}=z^{A}$ as $\Gamma \lbrack
x^{A},(x^{p})^{A},y^{A}]=C_{3}.$ Therefore every element in $S\setminus Z(S)$ is
of order $p.$ But this is not possible since $xy\in S\setminus Z(S)$ and is
of order $p^{2}.$ Thus we get $\left\vert Z(S)\right\vert =p.$

Let $L/M$ be an $A$-invariant minimal normal subgroup of $G/M.$ Suppose that 
$L/M$ is nonsolvable. Then by parts $(a)$ and $(c)$ of Proposition \ref%
{Propo 3} $L/M$ is nonabelian simple and $z^{A}\cap L\subset M$ and hence $%
z^{A}\cap L=\emptyset .$ It follows by $(b)$ of Proposition \ref{Propo 3}
that $L$ is a nonsolvable Eppo-group which has a nontrivial normal subgroup $%
M$. Appealing to Theorem \ref{Theo 1} we find $M=O_{2}(L)$. Thus $q=2\neq p$ which
yields that a Sylow $2$-subgroup of $L$ is elementary abelian. Then we see
that the kernel of $L/M$ on $O_{2}(L)=M=F(G)$ contains a Sylow $2$-subgroup
and hence is nontrivial. This contradiction shows that $L/M$ is solvable. As 
$M=O_{p^{\prime }}(G)$ we see that $L/M$ is an elementary abelian $p$-group.
Then $U=S\cap L$\ is a Sylow $p$-subgroup of \thinspace $L$ and $G=MN_{G}(U)$%
. If possible choose $r\in \pi (N_{G}(U))\setminus \{p\}$ and let $R$ be a
subgroup of $N_{G}(U))$ of order $r$. Since $Z(S)\leq U$ we see by Step 2
that $C_{G}(U)$ is a $p$-group. More precisely $Z(S)=[Z(S),R]\leq $ $[U,R]$
is nontrivial and $R$ acts Frobeniusly on $[Z(U),R]M$ as $r\neq q$ because $%
M $ is a Sylow $q$-subgroup of $G.$ This contradiction shows that $%
N_{G}(U)=S $ is a $p$-group and $G=MS$ and $L=MZ(S).$\\

\textit{Step 2. For any\thinspace } $S\in Syl_{p}(G)$\textit{\ we have }$%
z\in X=\{x\in S\setminus Z(S):x^{p}=1\}$\textit{\ is contained in an }$A$%
\textit{-orbit while }$S\setminus (Z(S)\cup X)$ \textit{is contained in
another }$A$\textit{-orbit, and one of the following holds :}

$(a)$ $S$\textit{\ is extraspecial,}

$(b)$ $Z_{2}(S)=S^{\prime }=\Phi (S)$\textit{\ is an extraspecial group of
exponent }$p$, $p$\textit{\ is odd.}\\

\textit{Proof.} Let $S\in Syl_{p}(G)$. By Step 1 we know that $U=$ $Z(S)$ is
of order $p$ and $U\neq S.$ The number $k$ of $A$-orbits the union of which
covers $S\setminus Z(S)$ is at most $2$ because otherwise there exist $x_{i}$
in $S\setminus Z(S)$ such that $x_{i}^{A},i=1,2,3$ and $u^{A}$ for any $%
1\neq u\in U$ are pairwise distinct. As $u^{A}$ is adjacent to $%
x_{i}^{A},i=1,2,3$ we see that $u^{A}=z^{A}$ which is not possible. Suppose
that $k=2.$ There exists $x\in S\setminus U$ such that $x\notin z^{A}$.
Assume that $x$ is also of order $p$. As $U$ acts fixed point freely on $M$
we see that 
\begin{equation*}
M=\left\langle C_{M}(y):y\in \left\langle x\right\rangle U\setminus
\{1\}\right\rangle =\left\langle C_{M}(t):t\in Ux\right\rangle .
\end{equation*}%
If $x\in Z_{2}(S)$ then $\left\langle x\right\rangle U\trianglelefteq S$ and
hence $x^{S}=xU$. Thus we obtain that $C_{M}(x)\neq 1$ and hence $x\in
z^{A}. $ If $x\notin Z_{2}(S)$ then as $k=2$ we see that $S\setminus
Z_{2}(S)\subset x^{A}$ and $Z_{2}(S)\setminus U\subset z^{A}.$ But then
there exists $u\in U$ such that $t=xu$ centralizes some nontrivial element
in $M$ implying the contradiction that $t\in (S\setminus Z_{2}(S))\cap
z^{A}. $ This shows that in case $k=2$ any element in $%
(S\setminus U)\setminus z^{A}$ is of order $p^{2}$ and hence any two
elements of $S\setminus Z(S)$ of the same order are lying in the same $A$%
-orbit. If $k=1 $ then $S\setminus Z(S)\subseteq z^{A}.$ In any case we have $%
\{x\in S\setminus Z(S):x^{p}=1\}$ $\subseteq z^{A}\ $and $S\setminus
(Z(S)\cup X)$ is contained in another $A$-orbit which consists of elements of
order $p^{2}$ if it is nonempty.

Suppose that there exists an element $y\in S$ of order $p^{2}$ such that $%
y^{p}\notin U=\left\langle u\right\rangle $. Then $y^{p}=v\in z^{A}.$ The
abelian group $\left\langle u,y\right\rangle $ acts on the nontrivial group $%
C_{M}(v).$ If $x$ $\in $ $\left\langle u,y\right\rangle \setminus
\left\langle v\right\rangle $ has some nontrivial fixed point in $C_{M}(v)$
then it must be $A$-conjugate to $z$ and hence is contained in $\Omega
_{1}(\left\langle u,y\right\rangle )\leq $ $\left\langle u,v\right\rangle $.
But clearly any element of $\left\langle u,v\right\rangle \setminus
\left\langle v\right\rangle $ acts fixed poiny freely on $C_{M}(v).$ It follows that $C_{M}(v)\left\langle u,y\right\rangle /\left\langle
v\right\rangle $ is a Frobenius group which is not possible as $\left\langle
u,y\right\rangle/\left\langle v\right\rangle $ is noncyclic of order $p^{2}$%
. This contradiction shows that $y^{p}\in \left\langle u\right\rangle $,
that is $\mho ^{1}(S)=U.$ So independent of the exponent of $S$ we have $%
\mho ^{1}(S)\leq U.$

If $S^{\prime }=U$ \ then we have $\Phi (S)=S^{\prime }\mho ^{1}(S)=U=Z(S)$
and hence $S$ is extraspecial and clearly $S=Z_{2}(S)$.

If $U<S^{\prime }$ then $S^{\prime }\setminus U$ and $S\setminus S^{\prime }$
are lying in different $A$-orbits. So there are no $A$-invariant subgroups $%
Y $ such that $MU<Y<MS=G$ except $MS^{\prime }.$ In particular we have $%
Z_{2}(S)=S^{\prime }$ or $Z_{2}(S)=S.$ As $S/U$ is of exponent $p$ and $%
Z_{2}(S)^{\prime }\leq U<S^{\prime }$ we see that $Z_{2}(S)=S^{\prime }=\Phi
(S)$. Similarly, as $MC_{S}(S^{\prime })$ is $A$-invariant and contains $MU$
we see that either $C_{S}(S^{\prime })\leq S^{\prime }$ or $S^{\prime
}<C_{S}(S^{\prime })S^{\prime }$ and hence $C_{S}(S^{\prime })S^{\prime }=S.$
The second implies $C_{S}(S^{\prime })=S$ and hence the contradiction $%
S^{\prime }\leq Z(S)=U.$ Thus we have $C_{S}(S^{\prime })\leq S^{\prime }.$
Suppose next that $z^{A}\cap S^{\prime }=\emptyset $. Then $MS^{\prime }$ is
a Frobenius group and hence $S^{\prime }$ is isomorphic to either $\mathbb{Z}%
_{p^{2}}$ or $Q_{8}$. So the unique subgroup of order $p$ in $%
S^{\prime }$ must be $A$-conjugate to $\left\langle z\right\rangle $ which
is not possible. This shows that $S^{\prime }\setminus U\subset z^{A}$ \ and
hence $S\setminus S^{\prime }$ is the set of elements of order $p^{2}$ in $S$%
.

Furthermore since $\exp (S^{\prime })=p$ we have that $S^{\prime }$ is
either extraspecial of exponent $p$ or is elementary abelian. Now $M$ is an
irreducible $GA$-module and $MS^{\prime }$ is a normal subgroup of $GA.$ Let 
$M=W_{1}\oplus \cdots \oplus W_{m}$ be the Wedderburn decomposition of $M$
with respect to $MS^{\prime }.$ As $A$ acts on the set $\{W_{1},\ldots
,W_{m}\}$ of $MS^{\prime }$-Wedderburn components, it leaves $W_{1}\cup \cdots \cup
W_{m}$ invariant. Since $M\setminus \{1\}$ is an $A$-orbit by Step 1 we must
have $W_{1}\cup \cdots \cup W_{m}=M$ and this is possible only if $m=1.$
Thus $M$ is a homogeneous $MS^{\prime }$-module and hence a homogeneous $%
MS^{\prime }/M$-module. If $S^{\prime }$ is abelian we obtain that $%
S^{\prime }/C_{S^{\prime }}(M)$ is cyclic of order $p\ $which is impossible
as $C_{S^{\prime }}(M)=1$ and $U<S^{\prime }$. So $S^{\prime }$ is
extraspecial of exponent $p$ which yields in particular that $p$ is odd.\\

\textit{Step 3. The end of the proof.}\medskip\ \ \ \ \ \ \ 

\textit{Proof.} Note that $S\cong \overline{G}\trianglelefteq $ $\overline{H}
$. Then $\overline{H}$ is isomorphic to a subgroup of $GL(m,q)=Aut(M)$,
which acts transitively on the set of all nonidentity elements of $M$ and
has a normal subgroup $\overline{G}$ which is isomorphic to $S$. Appealing
to \cite{Huppert} if $A$ is assumed to be solvable, or to \cite{He} which
uses CFSG we conclude that one of the following holds:

\begin{enumerate}
\item $SL(k,q^{n})\leq\overline{H}\leq\Gamma L(k,q^{n})$,

\item $Sp(k,q^{n})\trianglelefteq\overline{H}$, where the parameters $n$ and 
$k$ are related to the dimension $m$ by $m=nk$,

\item There exists a normal subgroup $\overline{N}$ in $\overline{H}$ which
is an extraspecial $2$-group of order $2^{m+1}$ such that $C_{\overline{H}}(%
\overline{N})=Z(\overline{H})$ and $\overline{H}/\overline{N}Z(\overline{H})$
is faithfully represented on $\overline{N}/Z(\overline{N})$ and

\begin{itemize}
\item either $m=2$,\; $q\in\{3,5,7,11,23\}$ and $\overline{H}/\overline{N}Z(%
\overline{H}) $ is isomorphic to a subgroup of $S_{3}$,

\item or $m=4$,\; $q=3$ and $\mathbb{Z}_{5}\leq\overline{H}/\overline{N}$ is
isomorphic to a subgroup of a Frobenius group of order $20$ acting
faithfully on $\overline{N}/Z(\overline{N})$.
\end{itemize}
\end{enumerate}

Assume that (1) or (2) holds. Then the fact that $\overline{H}$ has a normal
subgroup $\overline{G}$ isomorphic to $S$ bounds the parameters: $k\leq2$; and
$q^{n}\in\{2,3\}$ if $k=2$ , because otherwise $\overline{H}$ does not contain
a nonabelian normal $p$-subgroup. Suppose first that $k=1$. Then $\overline
{H}$ is isomorphic to a subgroup of $\Gamma L(1,q^{n})$. On the other
hand$\ \Gamma L(1,q^{n})=KL$ where $K$ is a normal subgroup which is
isomorphic to the multiplicative group of the field $GF(q^{n})$ and $L\cong
Aut(GF(q^{n})$, $K\trianglelefteq\Gamma L(1,q^{n})$ and $K\cap L=1.$ Note that
$K$ is cyclic of order $q^{n}-1$ and $L$ is cyclic of order $n.$ We deduce
that $\overline{G}$ has a cyclic normal subgroup which we denote by abuse of
language by $\overline{G}\cap K$ so that $\overline{G}/(\overline{G}\cap K)$
is also a cyclic group. Since on the other hand either $S$ is extraspecial or
$S^{\prime}$ is extraspecial we see that $S\cong\overline{G}$ must be an extraspecial group of
order $p^{3}$ containing a cyclic subgroup of order $p^{2}$ and also a
subgroup of order p different from $Z(\overline{G}).$ In any case there exists
an element $x\in\overline{G}\setminus (\overline{G}\cap K).$ If we  consider $x$ as an
element of $\Gamma L(1,q^{n})$ it can be given as  $x=h\gamma$ where
$h\in\overline{H}\cap K$ and $\left\langle \gamma\right\rangle $ is the unique
subgroup in $L$ of order $p.$ In particular, $p$ divides $n\ $\ and $\gamma$
is the automorphism of the field $GF(q^{n})$ given by $y^{\gamma}=y^{q^{r}}$
where $n=pr.$ Using this description we want to compute $C_{\overline{H}%
}(\overline{G}):$ We have
\[
\;\;C_{\overline{H}\cap K}(\overline{G})=C_{\overline{H}\cap K}(x)\leq
C_{K}(\gamma)\cong GF(q^{r})^{\times}%
\]
and hence $\left\vert C\overline{_{H}}(\overline{G})\right\vert $ divides
$\left\vert L\right\vert \left\vert C_{K}(\gamma)\right\vert =pr$ $(q^{r}-1).$
Clearly $Z(\overline{G})$ is of order $p$ and is contained in $(\overline
{G}\cap K)\cap C_{\overline{H}\cap K}(x)$ which gives the additional
information that $p$ divides $q^{r}-1.$ $\overline{H}/C_{\overline{H}}(\overline{G})$ is isomorphic to a subgroup of
$Aut(\overline{G})$ containing
\[
Inn(\overline{G})=\overline{G}C_{\overline{H}}(\overline{G})/C_{\overline{H}%
}(\overline{G})\cong\overline{G}/Z(\overline{G})\cong%
\mathbb{Z}
_{p}\times%
\mathbb{Z}
_{p}.
\]
We know by \cite{W}  that $Aut(\overline{G})/Inn(\overline{G})$ is isomorphic
to $%
\mathbb{Z}
_{2}$ if $\ \overline{G}\cong D_{8}$ and to the Frobenius group of order
$p(p-1)$ if $p$ is odd. But taking the metacyclic structure of $\Gamma
L(1,q^{n})$ into account we see that the group of automorphisms induced by
$\overline{H}$ on $\overline{G}$ modulo inner automorphisms must be abelian
and must leave $\overline{G}\cap K$ invariant and hence $\left\vert
\overline{H}\right\vert $ divides $2^{4}r(q^{r}-1)$ if $p=2$, and divides
$p^{3}(p-1)r(q^{r}-1)$  if $p$ is odd as $%
\mathbb{Z}
_{p-1}\cong\overline{H}/C_{\overline{H}}(Z(\overline{G}))$ acts transitively
on $Z(\overline{G})\setminus \{1\}.$ As $M\setminus \{1\}$ is an $\overline{H}$-orbit we see that
$q^{pr}-1$ divides the order of $\overline{H}$. So we get that
\[
q^{2r}-1\text{ divides }2^{4}r(q^{r}-1),\text{ or }p\text{ is odd and }%
q^{rp}-1\text{ divides \ }p^{3}(p-1)r(q^{r}-1)
\]
where $p$ is a prime dividing $q^{r}-1.$ In particular
\[
q^{rp}-1\leq p^{4}r(q^{r}-1)\leq(q^{r}-1)^{5}r<q^{r5}r<q^{rp}%
\]
if $p>5$ which is clearly not possible. So $p\in\{2,3,5\}.$

Let us assume that $p\geq3$ and hence $n=pr\geq3.$ Then by Theorem
3.5 and Theorem 3.9 in  \cite{He} we see that if $(n,q)\neq(6,2)$ there exists a prime
divisor $s$ of  $q^{n}-1$ which does not divide $q^{k}-1$ for any $0<k<n.$
Thus $s$ divides $(p-1)r.$ If $s$ divides $p-1$ then $s=2$ and $q$ is odd and
$s$ divides $q-1$ which is not possible. So $s$ divides $r$ and hence we have
$q^{s}\equiv q(\operatorname{mod}s)$ and hence $q^{rp}\equiv q^{\frac{r}{s}%
	p}\equiv1(\operatorname{mod}s)$ which contradicts the definition of $s.$ So we are
left with the case $(n,q)\neq(6,2).$ In this case $q^{n}-1=63$, $p=3$ and
$r=2$ giving that $p^{3}(p-1)r(q^{r}-1)=3^{3}2^{2}(2^{2}-1)$ which shows that
$q^{n}-1$ does not divide $p^{3}(p-1)r(q^{r}-1).$  So we have $p=2$ and we get
 $q^{r}+1$ divides $2^{4}r$. This yields that $r=1$ and $q\in\{3,7\}.$
Then $\overline{H}=\Gamma L(1,q^{2})$ and $[\overline{H}:\overline{G}]=q-1$.

Suppose next that $k=2$. Then either $q^{n}=3$ or $q^{n}=2$ and $m=k$. Note that in the second case $\overline{H}\leq S_{3}$ which is not possible as
$\overline{G}$ is nonabelian. In the first case $SL(2,3)=Sp(2,3)\leq
\overline{H}\leq\Gamma L(2,3)$ and $\overline{G}\leq O_{2}(\overline{H})$. The group 
$\overline{G}$ is an extraspecial group of order $8$ containing the involution
$\overline{z}$ outside its center and hence $\overline{G}\cong D_{8}.$ This
forces $\overline{H}$ to be the semidihedral group of order $16$ containing a
subgroup isomorphic to $Q_{8}$ and acting regularly on $M.$ So $(i)$ holds in
this case.

Now we can assume that (3) holds. We obtain immediately that $p=2,$ because
otherwise $[\overline{G},\overline{N}]\leq \overline{G}\cap \overline{N}$ $=1
$ implying the contradiction that $\overline{G}\leq C_{\overline{H}}(%
\overline{N})=Z(\overline{H})$. Thus $\overline{G}\overline{N}$ $\leq O_{2}(%
\overline{H})$. If $\overline{H}/\overline{N}Z(\overline{H})$ is not a $2$%
-group then $O_{2}(\overline{H})\leq \overline{N}O_{2}(Z(\overline{H}))$. So
we have either $\overline{G}\leq \overline{N}O_{2}(Z(\overline{H}))$ or $%
\overline{H}=\overline{G}\overline{N}Z(\overline{H})$ and $|\overline{G}:%
\overline{G}\cap \overline{N}Z(\overline{H})|=2.$ The second case may occur
only if $m=2.$

We shall distinguish between the cases $m=2$ and $m=4.$ If $m=4$ then $%
q=3,\;Z(\overline{H})=Z(\overline{N})\cong \mathbb{Z}_{2}$, $\overline{N}%
=O_{2}(\overline{H})$ is extraspecial of order $2^{5}$ and $\left\vert
H/N\right\vert \in \{5,10,20\}$. Therefore $\overline{G}\leq \overline{N}$.
On the other hand an element of order $5$ in $\overline{H}$ acts irreducibly
on $\overline{N}/Z(\overline{N})$ and normalizes $\overline{G}$. This
implies that $\overline{G}=\overline{N}$ and also $S\cong \overline{G}$ is
an extraspecial group of order $32$. There are two isomorphism classes of
extraspecial groups of order $32$, which differ by the number of involutions
they contain. And only the one which is the central product of a dihedral
group and a quaternion group and contains exactly $20$ elements of order $4$
admits an automorphism of order $5$. This forces that $S\cong Q_{8}\ast D_{8}
$. As elements of the same order in $S\setminus Z(S)$ are $A$-conjugate we
see that the elements of $\overline{G}$ of order $4$ must be conjugate in $%
\overline{H}$. Since an element $x\in \overline{G}$ of order $4$ has exactly
two conjugates inside $\overline{G}$, namely $x$ and $x^{-1}$, we see that $|%
\overline{H}:\overline{G}|\in \{10,20\}$. Observe that all maximal abelian
subgroups of $S$ are isomorphic to $%
\mathbb{Z}
_{4}\times 
\mathbb{Z}
_{2}$ and hence for any $1\neq w\in M$ we have $\left\vert
C_{S}(w)\right\vert =2$ so that the subgroup $\overline{G}\overline{L}$ is
transitive on \thinspace $M\setminus \{1\}$ where $\overline{L}\in Syl_{5}(%
\overline{H}).$

If $m=2$ then by \cite{Huppert}  $\overline{N}\cong Q_{8}$ and it
acts irreducibly on $M$ implying that $Z(\overline{H})$ acts by scalars on $%
M.$ Hence $Z(\overline{H})$ is cyclic of order dividing $q-1$ and $\overline{H}/\overline {%
N}Z(\overline{H})$ is isomorphic to a subgroup of $S_{3}.$ As $Z(\overline {N%
})\leq Z(\overline{H})$ and a Sylow $2$-subgroup of $Z(\overline{H})$ is of
order $2$ for $q\in\{3,7,11,23\}$ it holds that $Z(\overline{H})=Z(\overline {N%
})\times O_{2^{\prime}}(Z(\overline{H}))$. As $Q_{8}\ncong\overline{G}$ we
have $\overline{G}\nleqq\overline{N}$ and therefore $\overline{G}%
\overline{N}=O_{2}(\overline{H})$ implying $\overline{H}=\overline{G}%
\overline{N}O_{2^{\prime}}(Z(\overline{H}))$. This gives that $q=3$, and $%
\overline{G}\equiv D_{8}$ and $\overline{H}=\overline{G}\overline{N}.$ So we are left with the case $q=5$.

Notice that $Z(\overline{H})$ is of order dividing $4$ in this case. As $24=\left\vert
M\right\vert -1$ divides $\left\vert \overline{H}\right\vert $ we see that $
O_{2}(\overline{H})=\overline{N}Z(\overline{H})$ and $\overline{G}\leq%
\overline{N}Z(\overline{H})$. Since $\overline{G}$ is extraspecial we get\ $%
\overline{G}\cong D_{8}$ and hence $\overline{N}Z(\overline{H})\cong
Q_{8}\ast%
\mathbb{Z}
_{4}.$ Then an element of order $3$ in $\overline{H}$ must normalize $%
\overline{G}$ and hence must act trivially on it and on   $\overline{N}Z(\overline{H})$ as well. This contradiction
completes the proof.
\end{proof}

\begin{example}
We shall now present a slightly modified version of an example given in \cite%
{Huppert}. Consider the subgroup of $GL(4,3)$ generated by the matrices%
\newline

\ $\alpha =%
\begin{bmatrix}
N_{1} & O \\ 
O & N_{1}%
\end{bmatrix}%
$\,\,\, $\beta =%
\begin{bmatrix}
N_{2} & O \\ 
O & N_{2}%
\end{bmatrix}%
$\,\,\, $\gamma =%
\begin{bmatrix}
I_{2} & I_{2} \\ 
I_{2} & -I_{2}%
\end{bmatrix}%
$\,\,\, $\delta =%
\begin{bmatrix}
O & -I_{2} \\ 
I_{2} & O%
\end{bmatrix}%
$\ \ where\\

$O=%
\begin{bmatrix}
0 & 0 \\ 
0 & 0%
\end{bmatrix}%
$\,\,\, $I_{2}=%
\begin{bmatrix}
1 & 0 \\ 
0 & 1%
\end{bmatrix}%
$\,\,\, $N_{1}=%
\begin{bmatrix}
0 & 1 \\ 
-1 & 0%
\end{bmatrix}%
$\,\,\, $N_{2}=%
\begin{bmatrix}
1 & 0 \\ 
0 & -1%
\end{bmatrix}%
$\,\,\,

 $f=%
\begin{bmatrix}
1 & 1 & -1 & -1 \\ 
0 & 0 & -1 & 1 \\ 
0 & 0 & -1 & -1 \\ 
-1 & 1 & 1 & -1%
\end{bmatrix}%
$\,\,\, $g=$ $%
\begin{bmatrix}
0 & 1 & 0 & -1 \\ 
0 & 0 & 1 & 0 \\ 
0 & 1 & 0 & 1 \\ 
1 & 0 & 0 & 0%
\end{bmatrix}%
.$

Then we have the following relations:\newline
$\alpha^{2}=\gamma^{2}=\delta^{2}=-I_{4}=-\beta^{2},\;
\alpha^{\beta}=\alpha^{-1},\; \gamma^{\delta}=\gamma^{-1},\;$ $[\left\langle
\alpha,\beta\right\rangle ,\left\langle \gamma,\delta \right\rangle ]=I_{4},$
that is $\left\langle \alpha,\beta\right\rangle \cong D_{8},\; \left\langle
\gamma,\delta\right\rangle \cong Q_{8},$\; $\left\langle
\alpha,\beta,\gamma,\delta\right\rangle \cong D_{8}\ast Q_{8}$, the
extraspecial group of order $2^{5}$ with $20$ elements of order $4$.
Furthermore we have $f^{5}=-g^{4}=I_{4},$\; $f^{g}=f^{2},$\; $%
\alpha^{f}=\beta\gamma,$ $\beta^{f}=-\alpha\beta,$\; $\gamma^{f}=\alpha\beta%
\delta,$\; $\delta^{f}=-\gamma,$\; $\alpha^{g}=\alpha\beta\gamma,$\; $%
\beta^{g}=-\beta,$\; $\gamma ^{g}=-\beta\gamma\delta,$ $\delta^{g}=-\beta%
\delta.$

Set $B=\left\langle f,g\right\rangle $. Then $B/Z(B)$ is the Frobenius group
of order $20$ with $Z(B)=\left\langle -I_{4}\right\rangle $ and $B$ acts
nontrivially on $S=\left\langle \alpha,\beta,\gamma,\delta\right\rangle $ so
that the kernel of the action is $Z(B)=Z(S).$ The subgroup $SB$ of $GL(4,3)$
acts in a natural way on the vector space of column matrices over $GF(3)$
which we denote by $M$ and consider it as an elementary abelian group of
order $3^{4}.$ Let $G$ denote the normal subgroup $MS$ of the semidirect
product $A=MSB$ and consider $A$ as a group of automorphisms of $G.$ Observe
that if $u=-I_{4}\in Z(S)$ and $v=[1,0,0,0]^{T}\in C_{M}(\beta),$ then the $%
A $-orbits in $G\setminus \{1\}$ are represented by $v,u,\beta,v\beta,\delta$
which are of lengths $3^{4}-1, 3^{4}, 3^{2}\cdot10, 3^{2}(3^{2}-1)\cdot10,
3^{4}\cdot20$ respectively. Now if $\Gamma=\Gamma(G,A)$ we see that $\Gamma[
\beta^{A}, v^{A}, (v\beta)^{A}]=C_{3} =\Gamma[\beta^{A}, u^{A}, \delta^{A}]$
and $\Gamma$ is a friendship graph with two triangles joined at the singular
vertex $\beta^{A}.$
\end{example}

\begin{example}
Let $M$ be an elementary abelian group of order $3^{2}$ and $T$ be a
Sylow $2$-subgroup of $GL(2,3)=Aut(M).$ Note that $T\cap
SL(2,3)=\left\langle y_{1}, y_{2}\right\rangle $ is a quaternion group and
acts Frobeniusly on $M$. Let $z$ be an involution in $T\setminus SL(2,3).$
Then we have $T=\left\langle y_{1}, y_{2}, z\right\rangle $ and $z$ acts on $%
\langle y_{1}, y_{2}\rangle $ such that $y_{1}^{z}=y_{1}^{-1},\,\,
y_{2}^{z}=y_{1}y_{2}.$ Let $A=MT$ and $G=M\langle y_{1}, z\rangle.$ Then $G$
is a subgroup of $A$ of index $2$ on which $A$ acts faithfully by
conjugation. If $1\neq m\in C_{M}(z)$ then the $A$-orbits in $G\setminus
\{1\}$ are represented by $m, y_{1}^{2}, y_{1}, z, zm$ and are of lengths $%
8, 9, 18, 12, 24$ respectively. It holds that $\Gamma$ is an $\mathcal{F}$%
-graph consisting of a triangle $\Gamma[ z^{A}, m^{A}, (zm)^{A}]$ together
with a tail $\Gamma[ z^{A}, (y_{1}^{2})^{A}.y_{1}^{A}]= P_{2}.$
\end{example}

\section{Final Remarks}

\begin{corollaryn}
If $G$ is a group with $\left\vert \pi (G)\right\vert \geq 2$ and $\Gamma $
is an $\mathcal{F}$-graph for some $A\leq Inn(G)$ then $\left\vert
Z(G)\right\vert =2$ and $G/Z(G)\cong S_{3}$ and $A$ contains a Sylow $2$%
-subgroup of $Inn(G)\cong S_{3}$.
\end{corollaryn}

\begin{proof}
If $\Gamma $ has no singular vertex then Proposition \ref{Propo 1} shows
that $G$ is abelian of order at least $6$ which is not possible. Therefore $%
\Gamma $ possesses a singular vertex. Let $p$ be the order of a
representative of the singular vertex. If $O_{p}(G)=1$, it holds by Proposition \ref
{Propos 2} that $A$ is not contained in $Inn(G)$. Hence we may also
assume that $O_{p}(G)\neq 1.$

We shall repeatedly use the fact that if $\Gamma$ is an $\mathcal{F}$-graph
for some $A\leq Inn(G)$ and a Sylow subgroup $T$ of $G$ for some prime
different from $p$ is abelian then $N_{G}(T)$ must act transitively on $%
T\setminus \{1\}.$

Suppose first that $G$ is solvable. Then the Theorem \ref{Theorem} says that the
structure of $G$ can be described as $G=PQR$ where $P=O_{p}(G),$ $Q\in
Syl_{q}(G)$ for some prime $q\neq p,\,PQ\trianglelefteq G$ and $R\in
Syl_{r}(N_{G}(Q))$ where $r\ne q.$  We observe that $R=1 
$ implies that the elements of $Z(Q)\setminus \{1\}$ lie in different $A$-orbits and
hence that $\left\vert Q\right\vert =2$ . Therefore $p$ is odd. On the other hand 
$z\in P$ and $z$ is $A$-conjugate to all the other elements in $\left\langle
z\right\rangle \setminus \{1\}$ and this conjugation must take place in the
normalizer of $\left\langle z\right\rangle $. Note that the Sylow $p$-subgroup of $%
N_{G}(\left\langle z\right\rangle )$ centralizes $\left\langle
z\right\rangle $ and $z$ has to centralize an element of order $2$ and hence a
Sylow $2$-subgroup of $G$. This forces that $N_{G}(\left\langle z\right\rangle
)=C_{G}(z)$ which is impossible. Similarly we see that $
R/C_{R}(Q)$ is of order $r$ if it is nontrivial and if $r\neq p$ then $r=2$
and $\left\vert Q\right\vert =3$ and $p\geq 5.$ As $6$ has to divide $%
\left\vert C_{G}(z)\right\vert $ it holds that $C_{P}(QR)\neq 1.$ But $%
C_{P}(QR)\setminus \{1\}\subset z^{A}$ which is not possible because any two
elements of $C_{P}(QR)$ are conjugate in $G$ if and only if they are
conjugate in $N_{G}(C_{P}(QR))=N_{P}(C_{P}(QR))QR$ which has a center
containing at least $5$ elements. Therefore $R$ is a $p$-group and $\overline{G}%
=G/PQ$ contains only one subgroup of order $p$. If there exists a $p$-element $x$ such that $|\overline{x}|=p^{2}$ in $\overline{G}$
then it is not possible that $x$, $x^{-1}$, $x^{p}$\ represent different $A$
orbits. So we obtain that $\overline{G}\cong 
Q_{8}=\left\langle a,b\right\rangle .$ Then $a,b$ and $ab$ lie in
different $A$-orbits. It follows that $p=2=\left\vert R/C_{R}(Q)\right\vert $ and
hence $\left\vert Q\right\vert =3$. Notice that $A$ has to act transitively on $%
Y=\bigcup\nolimits_{X\in Syl_{q}(G)}C_{Z(P)}(X)\setminus \{1\}.$ We see that $Y=C_{Z(P)}(Q)\setminus \{1\}$ as $C_{Z(P)}(Q)$ is
normalized by $PQR=G$. If $Y\neq
\emptyset $ then $Y\subseteq Z(G)$ as $Y\cap Z(G)\neq \emptyset.$ This yields  $\left\vert C_{Z(P)}(Q)\right\vert \leq 2.$ If there exists an
involution $y$ in $[P,Q]\cap Z(S)$ where $S\in Syl_{2}(G)$ and $u\in
S\setminus P$ then $y\notin Y$ and hence $y,yz,u,z$ represent different $A$%
-orbits which are pairwise adjacent to each other. Then $[P,Q]=1$ and hence $%
P\subseteq z^{A}\cup \{1\}$. As $P\cap Z(G)\neq 1$ we get $%
P=Z(G)=\langle z\rangle \ $and $G/Z(G)\cong QR/C_{R}(Q)\cong S_{3}$. So
either $G\cong 
\mathbb{Z}
_{2}\times S_{3}$ or $G$ has cyclic $2$-subgroup of order $4$ and a normal
subgroup of order $3$ such that $Z(G)\cong 
\mathbb{Z}
_{2}$ and $G/Z(G)\cong S_{3}.$ Clearly a Sylow $2$-subgroup of $QR$ is
contained in $A$. Here $\Gamma $ is a friendship graph with four or two $C_{3}$
's, depending on whether $A$ is a $2$-group or not.

Assume next that $G$ is nonsolvable.  By the Theorem\ref{Theorem} we see that $G$ has a normal
subgroup such that $G/N$ is isomorphic to $PSL(2,5)$ or $PSL(2,7)$ or $%
PSL(3,4)$. In the last two of these groups a Sylow $2$-subgroup is not
abelian and so $p=2$. But then in both of these groups a Sylow $7$-group $S$ is
cyclic and in both of them $N_{G}(S)/C_{G}(S)$ is not of order $6.$ Therefore we
have $G/N\cong PSL(2,5).$ If $p\neq 5$ then $5$ does not divide $\left\vert
N\right\vert $ and hence $N_{G}(T)$ does not act transitively on $T\setminus
\{1\}$ where $T\in Syl_{5}(G)$. Then we must have $p=5$ and we are in the case
Theorem\ref{Theorem} (2)(c) or (3). In the former case $A$ is not contained in $Inn(G)$ and so
we have $G=P\times PSL(2,5)$ where $P=\left\langle z^{A}\right\rangle =Z(G)$
is an elementary abelian $5$-group. This contradiction completes the proof. 
\end{proof}

\begin{corollaryn}
For no nonabelian group $G$ with $\left\vert \pi(G)\right\vert \geq2$ the
commuting graph of $G$ is an $\mathcal{F}$-graph.
\end{corollaryn}

\begin{proof} Suppose that $G$ is a nonabelian group with $\left\vert \pi (G)\right\vert
	\geq 2.$ If $Z(G)=1$ then the commuting graph of $G$ is $\Gamma =\Gamma
	(G,1) $ and by the above corollary we see that it is not an $\mathcal{F}$%
	-graph. So $Z(G)$ is nontrivial. The commuting graph of $G$ is then $\Delta=\Gamma \lbrack G\setminus Z(G)]$.  For any  $x\in G\setminus Z(G)$ the graph $\Delta$ induces a clique on the vertex set $xZ(G)$ and hence  $|Z(G)|\leq 3.$ Furthermore if a vertex in $xZ(G)$ is adjacent to some vertex $v$ not contained in  $xZ(G)$ then every element of  $xZ(G)$ is adjacent to $v.$ This shows that $C_{G}(x)\setminus Z(G)=xZ(G)$ for any $x\in G\setminus Z(G)$ and  since $\Delta$ is connected we obtain $G=Z(G)\cup xZ(G)$ is an abelian group which is not possible.

\end{proof}

\begin{corollaryn}
$\Gamma $ is a starlike graph that is an $\mathcal{F}$-graph with no cycles
if and only if it is a star graph. If this is the case and $\Gamma $ has
at least $3$ vertices then $G$ is either a special $p$-group or is a $p$%
-group with $\Omega _{1}(Z(G))=Z(G)<G^{\prime }=Z_{2}(G)$ and $G^{\prime
\prime }=1$ and $C_{G}(y)\cap Z_{2}(G)=Z(G)$ for any $y\in G\setminus
Z_{2}(G).$ In both cases $Z_{2}(G)\setminus \{1\}$ is a union of two $A$%
-orbits.
\end{corollaryn}

\begin{proof}
If $\Gamma $ has no cycles, then it is a tree and so $G$ is a $p$-group for
some prime $p$ by Proposition \ref{Propo 2} (f)$.$ Therefore $\Gamma $ has a complete
vertex and has $\left\vert V\right\vert -1$ pendant vertices, that is, all
the rays of the graph are of length $1$. This proves the first claim. Assume now that $G$ is a $p$-group for some prime $p$ and has a group of
automorphisms $A$ such that $\Gamma $ is an $\mathcal{F}$-graph with no
cycles.

If $G$ is abelian then $\Gamma $ is a complete graph and hence has at most
two vertices. So we can assume that $G$ is nonabelian. This implies the existence of a vertex $x^{A}\subseteq G\setminus Z(G)$. Hence $Z(G)\setminus \{1\}=\{v\}$ is a
vertex and any other vertex is adjacent only to $v$. It follows that for any $x\in
G\setminus Z(G)$ and any $y\in C_{G}(x)\setminus Z(G)$ we have $x^{A}=y^{A},$
that is $C_{G}(x)\setminus Z(G)$ is an $A$-orbit for any $x\in G\setminus
Z(G).$ Now $G^{\prime }\leq C_{G}(x)\leq Z_{2}(G)$
for any $x\in Z_{2}(G)\setminus Z(G)$ because in this \ case we have $%
C_{G}(x)\trianglelefteq G$ and $G/C_{G}(x)$ is isomorphic to a subgroup of $%
Z(G)$ and both $C_{G}(x)$ and $Z_{2}(G)$ are $A$-invariant. Furthermore as $%
Z(G)\cap G^{\prime }\neq 1$ it holds that $Z(G)\leq G^{\prime }$. So have either $G^{\prime }=Z(G)=\Phi (G)$ as $G/Z(G)$ is of exponent $p$ or $%
Z(G)<G^{\prime }$ hence $G^{\prime }=C_{G}(x)$ for any $x\in
Z_{2}(G)\setminus Z(G).$ This implies in the second case that $G^{\prime
}=Z_{2}(G)=\{1\}\cup z^{A}\cup x^{A}$ for some $z\in Z(G)\,$and $x\in
G^{\prime }\setminus Z(G).$ In both cases we have $G^{\prime \prime }=1,$ $%
\Omega _{1}(Z(G))=Z(G)$ and if $y\in G\setminus Z_{2}(G)$ then $C_{G}(y)\cap
Z_{2}(G)=Z(G).$
\end{proof}

\begin{example}
Let $G=\left\langle a,b\right\rangle $ be the extraspecial group of order $%
3^{3}$ and exponent $3$ $.$ Let $B$ $=\left\langle x,t\right\rangle \cong
D_{8}$ acts on $G$ as follows: $x$ centralizes $Z(G)$, $t$ inverts $Z(G)$%
,\thinspace\ $a^{x}=b^{-1}$,\thinspace\ $b^{x}=a$,\thinspace\ $a^{t}=a^{-1}$%
,\thinspace\ $b^{t}=b$. Let $%
C=\langle x^{2},t\rangle $. Then $\Gamma (G,A_{1})=P_{3}$ where $%
A_{1}=Inn(G)B$\,\, and 
$\Gamma (G,A_{2})$ is a star graph with $4$ vertices where $A_{2}=Inn(G)C$.
\end{example}

\begin{corollaryn}
$\Gamma$ is a triangle free $\mathcal{F}$-graph if and only if it is a star
graph.
\end{corollaryn}

\begin{proof}
Suppose that $\Gamma$ is an $\mathcal{F}$-graph. To prove the claim, one
needs only to show that $\Gamma$ has no cycles in case it has no triangles. If $%
\Gamma$ has no singular vertex then $\Gamma=P_{n}$ for some $n\leq3$ by Proposition \ref{Propo 1} and hence is a star graph. If $%
\left\vert \pi(G)\right\vert \geq2$ then $\Gamma$ contains certainly a
triangle. Then $G$ is a $p$-group and hence $\Gamma$ has a complete vertex. 
It follows that $\Gamma$ has a triangle if it contains a cycle.
\end{proof}

\begin{corollaryn}
Suppose that $A$ acts coprimely on the group $G$ in such a way that $\Gamma$ is
an $\mathcal{F}$-graph. Then either $G$ is a $p$-group for some prime $p$, or 
$\Gamma=C_{3}$ and there are infinitely many different examples with $%
\Gamma=C_{3}$ and $(\left\vert G\right\vert ,\left\vert A\right\vert )=1.$
\end{corollaryn}

\begin{proof}
Suppose that $\left\vert \pi (G)\right\vert \geq 2$. Assume first that $%
\Gamma $ has an isolated vertex. Let $p$ be the order of a representative of
a complete vertex of $\Gamma .$ By the Theorem \ref{Theorem} we only  need to consider the case
that $G$ is solvable and $O_{p}(G)\neq 1$, because otherwise $2$ divides $%
(\left\vert G\right\vert ,\left\vert A\right\vert )$. Thus we may assume that $%
G/O_{p}(G)$ is a Frobenius group with an elementary abelian kernel $%
M/O_{p}(G)$ so that $A$ acts transitively on the set of nontrivial elements of $%
M/O_{p}(G)$. Then $\left\vert M/O_{p}(G)\right\vert -1$ divides $\left\vert
A\right\vert $. Every prime dividing the order of the
Frobenius complement of $G/O_{p}(G)$ divides also $%
\left\vert M/O_{p}(G)\right\vert -1$ and hence divides $(\left\vert
G\right\vert ,\left\vert A\right\vert )$. This shows that $\Gamma $ has no
isolated vertex. Now appealing to the Theorem\ref{Theorem} we see that $G=P\times Q$ where $P$ and $Q$ are
elementary abelian $p$- and $q$-groups for two distinct primes $p$ and $q.$ We also know that both $\left\vert P\right\vert -1$ and $\left\vert Q\right\vert
-1$ divide $\left\vert A\right\vert $
as $A$ acts transitively on both $P\setminus \{1\}$ and $Q\setminus \{1\}$. On the other hand there exists a
subgroup $A$ of $Aut(G)$ given as $A=A_{1}\times A_{2}$ \ with $\left\vert
A_{1}\right\vert =$ $\left\vert P\right\vert -1$ and $\left\vert
A_{2}\right\vert =\left\vert Q\right\vert -1$ so that $A_{1}$ acts
Frobeniusly on $P$ and centralizes $Q$, and $A_{2}$ acts Frobeniusly on $Q$
and centralizes $P$. So if the primes $p$ and $q$ and the positive
integers $n$ and $m$ are chosen in such a way that $p$ does not divide $%
q^{m}-1$ and $q$ does not divide $p^{n}-1$ then there exists a nilpotent
group $G$ of order $p^{n}q^{m}$ with elementary abelian Sylow subgroups and $%
A\leq Aut(G)$ such that $\Gamma=C_{3}$.
\end{proof}

Finally, we list some questions we would like to see them answered.\\

\begin{itemize}
\item [Q.1] What can be said about a $p$-group $G$ admitting a group $A$ of
automorphisms such that $\Gamma $ is an $\mathcal{F}$-graph?

\item [Q.2] Do the cases (b) and (c) in the Proposition \ref{Propo 7} really occur?

\item [Q.3] Do the quasisimple groups $2PSL(3,4)$ and $2^{2}PSL(3,4)$ have really
a group of automorphisms $A$ such that $\Gamma $ is an $\mathcal{F}$-graph?

\item [Q.4] What can be said about the groups $G$ such that $\Gamma $ is planar
for some $A\leq Aut(G)$?
\end{itemize}

\end{document}